\documentclass{article}

\usepackage{color}
 \catcode`@=11
\usepackage{color}
\usepackage{amsthm,amssymb,amsfonts,mathrsfs}
\usepackage{amsmath}
\catcode`@=11 \@addtoreset{equation}{section}

\catcode`@=12

\newtheorem{theorem}{Theorem}[section]
\newtheorem{proposition}{Proposition}[section]
\newtheorem{lemma}{Lemma}[section]
\newtheorem{corollary}{Corollary}[section]
\theoremstyle{definition}

\newtheorem{definition}{Definition}[section]
\newtheorem{remark}{Remark}

\newtheorem{assumptions}{Hypothesis}[section]

\newcommand{\ds}{\displaystyle}
\def\supp{\mathop{\rm supp}\nolimits}
\usepackage[colorlinks=true,urlcolor=red,linkcolor=blue,citecolor=red,bookmarks=red]{hyperref}
\title{Null controllability for a degenerate population equation with memory}
\author{{\sc Brahim Allal}\thanks{The author thanks the MAECI (Ministry of Foreign
Affairs and International Cooperation, Italy) for funding that
greatly facilitated scientific collaboration between Universit\'e Hassan $1^{er}$
(Morocco) and Universit\`{a} di Bari Aldo Moro (Italy).}\\
Facult\'e des Sciences et Techniques\\
Universit\'e Hassan 1er\\
Laboratoire MISI, B.P. 577\\
Settat 26000, Morocco
\\ email: b.allal@uhp.ac.ma
\\
{\sc Genni Fragnelli}\thanks{The author is a member of the Gruppo Nazionale per l'Analisi Ma\-te\-matica, la Probabilit\`a e le loro Applicazioni (GNAMPA) of the
Istituto Nazionale di Alta Matematica (INdAM) and she is supported by the FFABR {\it Fondo
per il finanziamento delle attivit\`a base di ricerca} 2017, by the INdAM - GNAMPA Project 2020 {\it Problemi inversi e di controllo per equazioni di evoluzione e loro applicazioni},  by  Fondi di Ateneo 2015/16, 2017/2018 of the University of  Bari {\em Problemi differenziali non linearii} and by PRIN 2017-2019 {\it Qualitative and quantitative aspects of nonlinear PDEs}.}\\
Dipartimento di Matematica\\ Universit\`{a} di Bari "Aldo Moro"\\
Via
E. Orabona 4\\ 70125 Bari - Italy\\ email: genni.fragnelli@uniba.it\\
{\sc  Jawad Salhi}\\
Moulay Ismail University of Meknes,\\
FST Errachidia, MAIS Laboratory, MAMCS Group,\\
P.O. Box 509, Boutalamine 52000, Errachidia, Morocco\\
email: sj.salhi@gmail.com}

\date{}
\begin{document}
\maketitle

\begin{abstract}
In this paper we consider the null controllability for a population model depending on time, on space and on age. Moreover, the diffusion coefficient degenerate at the boundary of the space domain. The novelty of this paper is that for the first time we consider the presence of a memory term, which makes the computations more difficult. However, under a suitable condition on the kernel we deduce a null controllability result for the original problem via new Carleman estimates for the adjoint problem associated to a suitable nonhomogeneous parabolic equation.
\end{abstract}
\section{Introduction}\label{Sect_Introduction}
In this work we are concerned with the null controllability result for a degenerate parabolic population equation with memory by a distributed control force. More precisely, we consider the following controlled system:
\begin{equation}\label{sys-memory1}
\left\{
\begin{array}{lll}
\displaystyle y_t  + y_a- (k(x)  y_x )_x + \mu(t,a,x) y= \int\limits_0^t b(t,s,a,x)  y(s,a,x) \, ds + f(t,a,x) \chi_\omega & \text{in } Q, \\
y(t,a,1)= y(t,a,0)=0, & \text{in } Q_{T,A},  \\
y(0,a,x)= y_{0}(a,x),  &  \text{in }  Q_{A,1},\\
y(t,0,x)=\int_0^A \beta(a,x) y(t,a,x) da, & \text{in } Q_{T,1},
\end{array}
\right.
\end{equation}
where $Q:=(0,T)\times(0,A)\times(0,1)$, $Q_{T,A} := (0,T)\times (0,A)$, $Q_{A,1}:=(0,A)\times(0,1)$ and
$Q_{T,1}:=(0,T)\times(0,1)$. Moreover, $y(t,a,x)$  is the distribution of certain individuals  at location $x \in (0,1)$, at time $t\in(0,T)$, where $T$ is fixed, and of age $a\in (0,A)$. $A$ is the maximal age of life, while $\beta$ and $\mu$ are the natural fertility and the natural death rates, respectively. Thus, the formula $\int_0^A \beta y da$ denotes the distribution of newborn individuals at time $t$ and location $x$; the function $b=b(t,s,a, x)\in L^\infty((0, T) \times Q)$ is a memory kernel. In the model $ \chi_\omega$ is the characteristic function of  the control region $\omega\subset(0,1)$ and $f$ is the control. Finally, $k$ is the dispersion coefficient and we assume that it depends on the space variable $x$ and degenerates at the boundary of the state space. In particular, we say that the function $k$ is 
\begin{definition}\label{Wdef} {\bf Weakly degenerate (WD)}  if $k \in  W^{1,1}([0,1])$,
\[
 k>0 \text{ in }(0,1)  \text{ and } k(0)=k(1)=0,
 \]
and there exist $M_1, M_2\in (0,1)$ such that
$x k'(x)  \le M_1k(x) $ and $(x-1)k'(x) \le M_2 k(x)$ for a.e. $x\in [0,1]$.
 \end{definition} 

or 
\begin{definition}\label{Sdef} {\bf Strongly degenerate (SD)}  if $k \in  W^{1,\infty}([0,1])$,
\[
 k>0 \text{ in }(0,1)  \text{ and } k(0)=k(1)=0,
 \]
and there exist $M_1, M_2\in [1,2)$ such that
$x k'(x)  \le M_1k(x) $ and $(x-1)k'(x) \le M_2 k(x)$ for a.e. $x\in [0,1]$.
 \end{definition} 
 For example, as $k$ one can consider $k(x)=x^{M_1}(1-x)^{M_2}$, with $M_1, M_2 >0$. Clearly, we say that $k$ is weakly or strongly degenerate only at $0$ if Definition \ref{Wdef} is satisfied only at $0$, i.e. $k \in  W^{1,1}([0,1])$, $
 k>0 \text{ in }(0,1], k(0)=0
$
and, there exists $M_1\in (0,1)$  or $M_1\in [1,2)$ such that
$x k'(x)  \le M_1k(x) $ for a.e.  $x\in [0,1]$.  Analogously at $1$.\medskip

Recently, population models have been widely investigated by many authors from many points of view. Mostly in the last period it is become clear how  the so called net reproduction rate $R_0$ affects the asymptotic behavior of the solution of the population. In particular, if $R_0>1$, the solution growths exponentially; if $R_0<1$, the solution decays exponentially; if $R_0=1$, the solution tends to the steady state solution.   Clearly, if $R_0>1$ and the system represents the distribution of a damaging insect population or of a pest population, it is very worrying.  It is enough to think what is happening in this period with the pandemic; this fact explains the reason why the governments try to keep the net reproduction rate below $1$. 
For this reason, even before the pandemic, great attention has been given to null controllability of damage population.

In the case that the dispersion coefficient 
{\it $k$ is a constant or a positive function} and the memory kernel is zero, null controllability for \eqref{sys-memory1} is studied, for example, in \cite{An} (see \cite{webb} for the well-posedness). If {\it $k$ degenerates at the boundary or at an interior point of the domain} and  again $b\equiv 0$ we refer, for example, to \cite{Alabau2006}, \cite{fm}, \cite{fm1} and, if $\mu$ is singular at the same point of $k$,  to \cite{fm2}, \cite{fm_hatvani}, \cite{fm_opuscola} and recently to \cite{yamamoto}. In particular, in \cite{yamamoto}, \eqref{sys-memory1} depends also on the size $\tau \in (\tau_1, \tau_2)$ and on a function $g(\tau)$, which is the growth modulus (i.e. $\int_{\tau_1}^{\tau_2} \frac{1}{ g(\tau)} d\tau$ is a spending time to grow the individual from size $\tau_1$ to size $\tau_2$). 

To our best knowledge, \cite{aem} is the first paper where  the dispersion coefficient can degenerate at the boundary of the domain (for example $k(x) = x^\alpha,$  being $x \in (0,1)$ and $\alpha >0$). Using Carleman estimates for the adjoint problem, the authors prove null controllability for \eqref{sys-memory1} under the condition $T\ge A$. 
The case $T<A$ is considered  in \cite{idriss}, \cite{em}, \cite{fJMPA} and \cite{fAnona}. In particular, in \cite{em} the problem is in {\it divergence form}, the function $k$ degenerates only at a point of the boundary and
the birth rate $\beta$ belongs to $C^2(Q)$ (necessary requirement in the proof of \cite [Proposition 4.2]{em}). A more general result is obtained in \cite{fAnona} where $\beta$ is only a continuous function, but $k$ can degenerate at both extremal points. In \cite{idriss} the problem is always  in {\it divergence form} and $k$ degenerates at an interior point $x_0$ and it belongs to $C[0,1]\cap C^1([0,1]\setminus\{x_0\})$; see also the recent paper \cite{f_irena}, where the functions are less regular and both cases $T<A$ and $T >A$ are considered.
The {\it non divergence form} is considered  in \cite{fJMPA}, where $k$ can degenerate at
 a one point of the boundary domain or in an interior point (for a cascade system we refer to \cite{bf}). We underline again that in all the previous papers, the memory kernel is zero.

The first  results on null controllability when $k=b=1$ and $y$ is independent of $a$ can be found in \cite{Guerrero2013}. In particular, S. Guerrero and O. Imanuvilov prove that \eqref{sys-memory1} fails to be null controllable with a boundary control. Indeed, there exists a set of initial states that cannot be driven to $0$ in any positive final time; see also\cite{Zhou2014} for a similar result if $b$ is a non-trivial constant. In \cite{Zhou2018} these results are then extended  to the context of one dimensional degenerate parabolic equation where $k(x) = x^{\alpha}$, with $x\in (0, 1)$ and $ 0\leq \alpha < 1 $, and it is proved that the null controllability of \eqref{sys-memory1} fails whereas the approximate property holds in a suitable state space with a boundary control acting at the extremity $x = 0$ or $x = 1$. In \cite{Lavanya2009, Saktiv2008}  the authors consider a nonlinear and non degenerate version of \eqref{sys-memory1} in the case that $y$ is independent of $a$ and proved that the problem is null controllable assuming that the memory kernel  is sufficiently smooth and vanishes at the neighborhood of initial and final times. This assumption has been relaxed by Q. Tao and H. Gao in  \cite{TG2016}; for related results on this subject, we refer to \cite{Zuazua2017} for wave equation, \cite{Barbu2000} for viscoelasticity equation, \cite{Munoz2003} for thermoelastic system and \cite{Yong2005} in the case of heat equation with hyperbolic memory kernel (see also the bibliography therein).

To our knowledge, the first results on null controllability for a degenerate equation similar to \eqref{sys-memory1} with $b$ different from zero and $y$ independent of $a$ can be found in \cite{Allal2020} (see also \cite{AFS2020} for a singular/degenerate equation). In particular, in \cite{Allal2020} assuming that the kernel $b$ satisfies

\[
	(T-t)^{k} e^{\frac{C}{(T-t)^4}} b \in L^{\infty}((0, T) \times Q),
	\]
for some positive constant $C$,
the authors prove that the system is null controllable.

The purpose of this paper is to give a suitable condition on the memory kernel  $b$ similar to the previous one in such a way that  the degenerate parabolic equation with memory \eqref{sys-memory1} is null controllable. We underline that this is the first paper where the diffusion function $k$ is degenerate, the memory $b$ is a general nontrivial function and $y$ depends, not only on $t$ and $x$, but also on the age $a$. Due to the presence of $a$, the method used in \cite{Allal2020} cannot be applied directly, but some modifications are required, so the result is not a simple adaptation of the previous results.

\vspace{0.5cm}
We include here a brief description of the proof strategy:
in a first step, we focus on the following nonhomogeneous degenerate parabolic system
\begin{equation}\label{sys-nonhom}
\begin{cases}
\displaystyle y_t  + y_a- (k(x)  y_x )_x + \mu(t,a,x) y = h+ \chi_{\omega} f,  & (t, x) \in Q, \\
y(t,a,1)= y(t,a,0)=0, & \text{in } Q_{T,A},  \\
y(0,a,x)= y_{0}(a,x),  &  \text{in }  Q_{A,1},\\
y(t,0,x)=\int_0^A \beta(a,x) y(t,a,x) da, & \text{in } Q_{T,1},
\end{cases}
\end{equation}
for a given function $h \in L^2(Q)$.

In particular, we establish suitable Carleman estimates for the associated adjoint problem using some classical weight time functions that blow up to $+\infty$ as $t\rightarrow 0^-, T^+$. Then, using a weight time function not exploding in the neighborhood of $t=0$, we derive a new  modified Carleman estimate that would allow us to show null controllability of the underlying parabolic equation. As a consequence, we deduce a null controllability result for the following problem
\begin{equation}\label{sys-nonhom_new}\left\{
\begin{array}{lll}
\displaystyle y_t  + y_a- (k(x)  y_x )_x + \mu(t,a,x) y= \int\limits_0^t b(t,s,a,x)  w(s,a,x) \, ds + f(t,a,x) \chi_\omega & \text{in } Q, \\
y(t,a,1)= y(t,a,0)=0, & \text{in } Q_{T,A},  \\
y(0,a,x)= y_{0}(a,x),  &  \text{in }  Q_{A,1},\\
y(t,0,x)=\int_0^A \beta(a,x) y(t,a,x) da, & \text{in } Q_{T,1},\\
\end{array}
\right.\end{equation}
for a fixed $w \in L^2(Q)$.
 Finally, this controllability result combined with an appropriate application of Kakutani's fixed point Theorem allows us to obtain the null controllability result for the original system \eqref{sys-memory1} under a suitable condition on the kernel $b$.  Observe that in this paper, as in \cite{fJMPA} or in \cite{fAnona}, we do not consider the positivity of the solution, even if it is clearly interesting. This will be the subject of a forthcoming paper.

The  paper is organized in the following way: in Section \ref{sect_well_posed} we consider the well-posedness of systems \eqref{sys-memory1} and \eqref{sys-nonhom} in suitable weighted spaces. In Section \ref{sect_carleman_estimate}, we develop a new Carleman estimate for the adjoint problem associated to the nonhomogeneous parabolic equation \eqref{sys-nonhom}. In Section \ref{sect_null_control_nonhomog}, we apply such an estimate to deduce null controllability for \eqref{sys-nonhom_new} and, as a consequence, the null controllability result for  \eqref{sys-memory1} using the Kakutani's fixed point Theorem. Finally, in Section \ref{appendix}, we give the proofs of some results given in Sections \ref{sect_well_posed} and \ref{sect_carleman_estimate} for the reader's convenience.

\section{Well-posedness results} \label{sect_well_posed} 
The goal of this section is to study the well-posedness results for \eqref{sys-memory1} and \eqref{sys-nonhom}. First, we recall the following weighted Sobolev spaces (in the sequel, a.c. means  absolutely continuous):
\[
\begin{aligned}
 H^1_k(0,1):=\{ u \in L^2(0,1) \ \mid \ u \text{ a.c. in } [0,1], \sqrt{k} u_x \in  L^2(0,1) \text{ and } u(1)=u(0)=0 \}
\end{aligned}
\]
and
\[
H^2_k(0,1):=  \{ u \in H^1_k(0,1) |\,ku_x \in
H^1(0,1)\},
\]
equipped with the norms
\begin{align*}
\| y \|_{H_k^1(0,1)}^2:=  \| y \|_{L^2(0,1)}^2 +  \| \sqrt{k} y_x \|_{L^2(0,1)}^2\quad
\text{ and } \quad \| y \|_{H_k^2(0,1)}^2 :=  \| y \|_{H_k^1(0,1)}^2 +  \| (k y_x )_x \|_{L^2(0,1)}^2.
\end{align*}
 As in \cite{Alabau2006} or \cite{fm1}, we have that the  operator
\[\mathcal A_0u:= (ku_{x})_x,\qquad    D(\mathcal A_0): = H^2_{k}(0,1)\]
is self--adjoint, nonpositive  and generates a strongly continuous
semigroup  on the space $L^2(0,1)$.

Now, setting $ \mathcal A_a u := \ds \frac{\partial  u}{\partial a}$, we have that
\[
\mathcal Au:= \mathcal A_a u - 
\mathcal A_0 u,
\]

for 
\[
u \in D(\mathcal A) =\left\{u \in L^2(0,A;D(\mathcal A_0)) : \frac{\partial u}{\partial a} \in  L^2(0,A;H^1_k(0,1)), u(0, x)= \int_0^A \beta(a, x) u(a, x) da\right\},
\]
generates a strongly continuous semigroup on $L^2(Q_{A,1}):= L^2(0,A; L^2(0,1))$ (see also \cite{iannelli}). Therefore, 
$
(\mathcal A + B(t), D(\mathcal A))
$ generates a strongly continuous semigroup. Here
$B(t)$ is defined as
\[
B(t) u:= \mu(t,a,x) u,
\]
for $u \in D(\mathcal A)$, thus it can be seen as a bounded perturbation of $\mathcal A$ (see, for example, \cite{Alabau2006}).

Setting $L^2(Q):= L^2(0,T;L^2(Q_{A,1}))$
and assuming the next hypothesis:
\begin{assumptions}\label{ratesAss}
 The  functions $b$,  $\mu$ and $\beta$ are such that
\begin{equation}\label{3}
\begin{aligned}
&\bullet b=b(t,s,a, x)\in L^\infty((0, T) \times Q),\\
&\bullet \beta \in  C(\bar Q_{A,1}) \text{ and } \beta \geq0  \text{ in } Q_{A,1}, \\
&\bullet \mu \in C(\bar Q) \text{ and }  \mu\geq0\text{ in } Q,
\end{aligned}
\end{equation}
\end{assumptions}
the following well posedness  result holds:
\begin{theorem}\label{theorem_existence}
Assume that $b$, $\mu$ anf $\beta$ satisfy Hypothesis $\ref{ratesAss}$ and suppose that $k$ is weakly or strongly degenerate at $0$ and/or at $1$.  For all $f, h \in
L^2(Q)$ and $y_0 \in L^2(Q_{A,1})$, the system \eqref{sys-nonhom} admits a unique solution 
\[y \in \mathcal W_T:= C\big([0,T];
L^2(Q_{A,1}))\big) \cap L^2 \big(0,T;
L^2(0,A; H^1_k(0,1))\big)\]
and 

\begin{align}\label{stimau}
\sup_{t \in [0,T]} \|y(t)\|^2_{L^2(Q_{A,1})} & +\int_0^T\int_0^A\|\sqrt{k}y_x\|^2_{L^2(0,1)}dadt\notag\\
& \le C\left( \|y_0\|^2_{L^2(Q_{A,1})}  + \|f\chi_\omega\|^2_{L^2(Q)}+ \|h\|^2_{L^2(Q)}\right),
\end{align}
where $C$ is a positive constant independent of $b, k, y_0, h$ and $f$. In addition, if $f\equiv h\equiv 0$,  then
$
y\in C^1\big([0,T];L^2(Q_{A,1})\big).
$ 
Moreover, if $y_0 \in L^2(0,A; H^1_k(0,1))$, then
\[y \in \mathcal Z_T:= L^2(0, T;L^2(0,A; H^1_k(0,1)))\cap H^1(0,T; L^2(Q_{A,1})) \] and \eqref{stimau} holds.
\end{theorem}
For the reader's convenience, we postpone the proof of the previous theorem to the Appendix. 

The following result establishes existence and uniqueness of solution for system \eqref{sys-memory1}.

\begin{theorem}\label{prop-Well-posed_memory}Assume Hypothesis $\ref{ratesAss}$ and suppose  that $k$ is weakly or strongly degenerate at $0$ and/or at $1$. If $y_0 \in L^2(Q_{A,1})$ and $f\in L^2(Q)$. Then, system \eqref{sys-memory1} admits a unique solution
$
y \in \mathcal W_T$.
\end{theorem}
\begin{proof}
First of all, we transform \eqref{sys-memory1} into the following Cauchy problem
\begin{equation}\label{Cauch_problem}
\left\{
\begin{array}{ll}
\displaystyle y'(t) + (\mathcal A+B(t))y(t) = \int\limits_0^th(t,s, y(s)) \, ds + g(t), \quad t \in (0, T), \\
y(0)= y_{0}, 
\end{array}
\right.
\end{equation}
where $\mathcal A$ and $B(\cdot)$ are defined as before, $g(t) := 1_{\omega} f(t), \quad \text{for a.e. } t \in (0, T), $ and
$$ h(t,s,y(s)) : =  b(t,s, \cdot)  y(s), \quad \text{for a.e. } (t,s) \in (0, T)\times (0,T).$$
Proceeding as in \cite{Allal2020}, the thesis follows by \cite[Theorem 1.1]{Grasselli1991}. 

\end{proof}
We emphasis that, in order to prove the main null controllability result for \eqref{sys-memory1} it is not necessary to require more regularity for the initial data. Indeed, we only need existence and uniqueness in the case $y_0 \in L^2(Q_{A,1})$.

\section{Carleman estimates} \label{sect_carleman_estimate}
The goal of this section is to establish suitable Carleman estimates for the following adjoint parabolic system 
\begin{equation}\label{sys-adj-nonhom}
\begin{cases}
\ds \frac{\partial v}{\partial t} + \frac{\partial v}{\partial a}
+(k(x)v_{x})_x-\mu(t, a, x)v +\beta(a,x)v(t,0,x) =g,& (t,x,a) \in  Q,
\\[5pt]
v(t,a,0)=v(t,a,1) =0, &(t,a) \in Q_{T,A},\\
  v(t,A,x)=0, & (t,x) \in Q_{T,1},
\end{cases}
\end{equation}
where $g\in L^2(Q)$ and the rates $\mu$ and $\beta$ satisfy Hypothesis \ref{ratesAss}.

Moreover, from now on, we assume that the control set $\omega$ is such that
\begin{equation}\label{omega}
\omega=  (\alpha, \rho)  \subset\subset  (0,1).
\end{equation}
\subsection{The case $k(0)=0$}

In this subsection we will consider the case when $k$ degenerates only at $x=0$. In particular, we make the following assumptions:
\begin{assumptions}\label{BAss01} The function
$k\in C^0[0,1]\bigcap C^1(0,1]$  is such that $k(0)=0$, $k>0$ on
$(0,1]$ and there exists $M_1\in (0,2)$ such that
$x k'(x)  \le M_1k(x) $ for a.e. $x \in [0,1]$. Moreover, if $M_1\ge1$ one has to require that there exists $\theta \in (0,M_1]$, such that the function $x \ds \mapsto \frac{k(x)}{x^\theta}$ is nondecreasing near $0$.
\end{assumptions}
Now, we consider, as in \cite{fAnona}, the following weight functions 
\begin{equation}\label{varphi}
\varphi(t,a,x):=\Theta(t,a)\psi(x),
\end{equation}
where
\begin{equation}\label{Theta} 
\Theta(t, a)= \frac{1}{t^{4}(T-t)^{4}a^{4}},
\end{equation}
\begin{equation}\label{psi} 
\psi(x):=p(x) - 2 \|p\|_{L^\infty(0,1)}
\end{equation}
and
$\displaystyle p(x):=\int_0^x\frac{y}{k(y)}dy$. Observe that $\psi$ is increasing; in particular, $\psi(0)=-2\|p\|_{L^\infty(0,1)}$ and $\psi(1)=\int_0^1 \frac{x}{k(x)}dx-2\|p\|_{L^\infty(0,1)}= -\|p\|_{L^\infty(0,1)}$.
Clearly  $\psi <0$, thus $ \varphi(t,a, x)  <0$ for all $(t,a,x) \in Q$ and
$\varphi(t,a, x)  \rightarrow - \infty \, \text{ as } t \rightarrow
0^+, T^-$  or  $a \rightarrow
0^+$.

We also define
\begin{equation}
\eta(t,a, x)=\Theta(t, a) \Psi(x), \quad  \Psi(x)=e^{\kappa \rho(x)}-e^{2 \kappa\|\rho\|_{\infty}},
\end{equation}
$\ds
(t, a, x) \in Q, \quad \kappa>0 \quad\text {and} \quad \rho(x):= \mathfrak{d} \int_{x}^{1} \frac{1}{k(t)} d t, \quad\text { where} \quad\mathfrak{d}=\left\|k^{\prime}\right\|_{L^{\infty}(0,1)}$.
By taking the parameter $\kappa$ such that
\begin{equation}\label{cond_kappap}
\ds \kappa \leq \frac{\ln\left(\|p\|_{L^\infty(0,1)}+ 1\right)}{2\|\rho\|_{\infty}}
\end{equation}
we have  $\max\limits_{x \in [0,1]}\psi(x) \le \min\limits_{x \in [0,1]}\Psi(x)$,
 thus
\begin{equation}\label{compar_varphi_eta}
\varphi(t,a,x) \leq \eta(t,a,x), \;\, \text{for all} \;(t,a,x) \in Q.
\end{equation}
The following Carleman estimate holds.
\begin{theorem}\label{Cor1}
Assume Hypotheses $\ref{ratesAss}$ and $\ref{BAss01}$. Then,
there exist two strictly positive constants $C$ and $s_0$ such that every
solution 
$z \in
\mathcal Z_T$
of
\begin{equation} \label{adjoint}
\begin{cases}
\ds \frac{\partial z}{\partial t} + \frac{\partial z}{\partial a}
+(k(x)z_{x})_x-\mu(t, a, x)z =g,& (t,a,x) \in Q,\\
  z(t, a, 0)=z(t, a, 1)=0, & (t,a) \in Q_{T,A},\\
   z(t,A,x)=0, & (t,x) \in Q_{T,1}
\end{cases}
\end{equation}
satisfies, for all $s \ge s_0$,

\begin{align*}
\int_{Q} \Big(s \Theta k  z_x^2 & + s^3\Theta^3  \frac{x^2}{k} z^2\Big)e^{2s\varphi} \, dxdadt \\
&\leq C \Big(\int_{Q}g^{2} e^{2s\eta} dxdadt+
 \int_0^T \int_0^A\int_{\omega} s^2 \Theta^2 z^2 e^{2s\eta}  dx dadt\Big).
\end{align*}
\end{theorem}
\begin{proof}
The proof of this theorem mainly follows the same ideas of \cite[Theorem 4.2]{fAnona}. The only difference to point out is the following: instead of applying  Caccioppoli's inequality given in \cite[Proposition 4.2]{fAnona}, we apply the following Lemma (whose proof is postpone to the Appendix).

\begin{lemma}\label{lemma_caccio}
Let $\omega^{\prime}$ and $\omega$ two open subintervals of (0,1) such that $\omega^{\prime} \Subset \omega \Subset (0,1) .$ Let $\psi(t, a, x):=\Theta(t, a) \phi(x),$ where $\Theta$ is defined in \eqref{Theta} and $\phi \in C^{2}(0,1)$ is a strictly negative function. Then, there exist two strictly positive constants $C$ and $s_{0}$ such that, for all $s \geq s_{0}$,  every solution $z$ of \eqref{adjoint} satisfies
\begin{equation}\label{inequality_caccio}
\int_{0}^{T} \int_{0}^{A} \int_{\omega^{\prime}} z_{x}^{2} e^{2 s \psi} d x d a d t \leq C\left(\int_{0}^{T} \int_{0}^{A} \int_{\omega}  ( s^2 \Theta^2 z^{2} +  g^{2} ) e^{2 s \psi} d x d a d t\right).
\end{equation}
\end{lemma}
Hence, one can easily deduce the desired estimate given in Theorem \ref{Cor1}.
\end{proof}

As a consequence of the previous theorem, one has the next result.
\begin{corollary}\label{Cor}
Assume Hypotheses $\ref{ratesAss}$ and $\ref{BAss01}$. Then,
there exist two strictly positive constants $C$ and $s_0$ such that every
solution $v \in
\mathcal Z_T$ of \eqref{sys-adj-nonhom}
satisfies, for all $s \ge s_0$,
\begin{align}\label{Carl_estimate-2}
\int_{Q} \Big(s \Theta k  v_x^2 & + s^3\Theta^3  \frac{x^2}{k} v^2\Big)e^{2s\varphi} \, dxdadt \notag\\
&\leq C \Big(\int_{Q}g^{2} e^{2s\eta} dxdadt + \int_Q v^2(t,0,x) e^{2s\eta} dxdadt \Big)\notag\\
&+ C \int_0^T \int_0^A\int_{\omega} s^2 \Theta^2 v^2 e^{2s\eta}  dx dadt.
\end{align}
\end{corollary}
\begin{proof}
Rewrite the equation of
   \eqref{sys-adj-nonhom} as $ \ds\frac{\partial v}{\partial t} + \ds \frac{\partial v}{\partial a}
+(k(x)v_{x})_x-\mu(t, a, x)v =\bar g, $ where $\bar{g}
    := g -\beta(a,x) v(t,0,x)$. Then, applying Theorem \ref{Cor1}, there exists
    two positive constants $C$ and $s_0 >0$, such that
    \begin{equation*}
    \begin{aligned}
\int_{Q}\left(s \Theta k v_x^2
                + s^3\Theta^3\text{\small$\displaystyle\frac{x^2}{k}$\normalsize}
                  v^2\right)e^{2s\varphi}dxdadt
&\le
C\int_{Q}\bar g^{2}\text{\small$e^{2s\eta}$\normalsize}dxdadt
\\&+ C\int_0^T \int_0^A\int_ \omega s^2 \Theta^2 v^2 e^{2s\eta} dx dadt\\
& \le C\int_{Q}g^{2}\text{\small$e^{2s\eta}$\normalsize}dxdadt + C \int_Q v^2(t,0,x) e^{2s\eta} dxdadt
\\&+ C\int_0^T \int_0^A\int_ \omega s^2 \Theta^2 v^2 e^{2s\eta} dx dadt
   \end{aligned}
    \end{equation*}
    for all $s \ge s_0$.
\end{proof}
Observe that, if the problem is considered in $(T_1, T_2)\times(a_0,A)\times (0,1)$ with $T_1, T_2, a_0>0$, the previous result still holds considering $\Theta(t,a):=\frac{1}{[(t- T_1)(T_2-t)]^4(a-a_0)^4}$ for all $(t,a) \in (T_1, T_2)\times(a_0, A)$.

Following \cite{Allal2020}, by \eqref{Carl_estimate-2} we will derive a new modified Carleman inequality with a weight time function exploding only at the final time $t=T$. This new weight allows us to derive a null controllability result for system \eqref{sys-memory1} imposing a restriction on the kernel $b$ only at the final time $t=T$. 
To this end, let us introduce the following weight functions:
\begin{equation}\label{modif_weight_funct}
\begin{aligned}
&\gamma(t,a):=
\begin{cases}
\Theta(\frac{T}{2}, a) = \big(\frac{4}{T^2 }\big)^4\frac{1}{a^4},   & \text{for} \; t\in\big[0,\frac{T}{2}\big],\\
\Theta(t,a), & \text{for} \; t\in \big[\frac{T}{2}, T\big],
\end{cases} \quad \forall\; a \in (0,A), \\
& \Phi(t,a,x) :=  \gamma(t,a)\psi(x), \qquad \sigma(t,a,x): = \gamma(t,a)\Psi(x),  \quad \forall\; (t,a, x) \in Q, \\
& \widehat{\Phi}(t):= \min_{a\in(0,A]}\max\limits_{x\in[0,1]} \Phi(t,a,x)= \gamma(t,A)\psi(1),  \quad \forall \;t \in (0,T),\\
& \Phi^*(t, a):=\min\limits_{x\in[0,1]} \Phi(t,a,x)= \gamma(t, a)\psi(0),  \quad \forall\; (t,a) \in Q_{T,A}.
\end{aligned}
\end{equation}
Now, we are ready to state the following modified Carleman estimate, which reveals to be a major tool to obtain the null controllability result of the memory system \eqref{sys-memory1}.
 \begin{proposition}\label{Prop_modif_Carl_estimate_1} Assume Hypotheses $\ref{ratesAss}$, $\ref{BAss01}$ and let $T_0 \in [T/2, T)$. Then, for all $\delta \in (0,A)$,
there exist two positive constants $s_0$  and $C=C(s, T_0, \delta)$ such that every solution $v \in \mathcal Z_T$ of \eqref{sys-adj-nonhom} satisfies
\begin{equation}\label{modif_Carl_estimate}
\begin{aligned}
&\| e^{s\widehat{\Phi}(0)} v(0) \|_{L^2(Q_{A,1})}^2  + \int_Q e^{2s\Phi} v^2   \,dxdadt  \leq C \bigg(\int_{Q}e^{2s\hat \Phi(0)} g^{2} dxdadt \\
&+ \int_0^T \int_0^1 e^{2s\hat \Phi(0)} v^2(t,0,x)  dxdt + \int_0^T \int_0^A\int_{\omega} s^2 \gamma^2 v^2 e^{2s\sigma} dx dadt + \int_{T_0}^{T}\int_0^{\delta}\int_0^1 e^{2s\hat \Phi(0)} v^2 dxdadt\bigg)
\end{aligned}
\end{equation}
for all $s \geq s_0$. 
\end{proposition}
\begin{proof} Fix $T^* \in (T_0, T)$
and $ \xi \in C^{\infty}[0, T]$ be a smooth cut-off function such that
\begin{equation}\label{Xi}
\begin{aligned} 
\left\{
\begin{array}{lll}
0\leq \xi(t) \leq 1, & \text{for} \; t\in [0, T],\\
\xi(t)=1, & t\in\big[0,T_0\big],\\
\xi(t)=0, &  \; t\in \big[T^*, T\big]
\end{array}
\right. 
\end{aligned}
\end{equation}
and define $ w = \tilde{\xi} v$,  where $ \tilde{\xi} =  \xi  e^{s\widehat{\Phi}(0)}$, where $v $ solves \eqref{sys-adj-nonhom}.

Hence $w$ satisfies
\begin{equation}\label{sys-adj-nonhom-w}
\begin{cases}
\ds \frac{\partial w}{\partial t} + \frac{\partial w}{\partial a}
+(k(x)w_{x})_x-\mu(t, a, x)w =\hat{g},& (t,x,a) \in  Q,
\\[5pt]
w(t,a,0)=w(t,a,1) =0, &(t,a) \in Q_{T,A},\\
w(T,a,x) = 0, &(a,x) \in Q_{A,1}, \\
  w(t,A,x)=0, & (t,x) \in Q_{T,1}
\end{cases}
\end{equation}
where $\hat{g}:= \tilde \xi_t v + \tilde \xi g  - \beta(a,x)w(t,0,x)$.

By \eqref{stimau} applied to the above system, there exists a positive constant $C$ such that
\begin{equation*}
\sup_{t \in [0,T]} \|w(t)\|^2_{L^2(Q_{A,1})} \leq C \int_{Q}\hat{g}^2  \,dxdadt,
\end{equation*}
which yields
\begin{equation}\label{energy-w}
\begin{aligned}
&\|w(0)\|^2_{L^2(Q_{A,1})} +\|w\|^2_{L^2(Q)} \leq C  \int_{Q} \left(\tilde{\xi}_t v  + \tilde{\xi} g- \beta(a,x)  \xi  e^{s\widehat{\Phi}(0)} v (t,0,x)\right)^2  \,dxdadt.
\end{aligned}
\end{equation}
Moreover, one can see that
\begin{equation}\label{left-enrgy-1}
\| w(0)\|_{L^2(Q_{A,1})}^2 = \| e^{s\widehat{\Phi}(0)} v (0)\|_{L^2(Q_{A,1})}^2 
\end{equation}
and
\begin{equation}\label{left-enrgy-w-2}
\begin{aligned}
 \|w\|^2_{L^2(Q)} &=   \int_{0}^{T^*}\int_0^A\int_0^1\xi^2  e^{2s\widehat{\Phi}(0)} v ^2 dxdadt \\
 &\geq \int_{0}^{T_0}\int_0^A\int_0^1 v ^2 e^{2s\Phi}  \,dxdadt 
\end{aligned}
\end{equation}
since $ \Phi \leq \widehat{\Phi}(0) \; \text{in} \; Q$. 

We also have
\begin{equation}\label{right-enrgy-w}
\begin{aligned}
&\int_{Q}  \left( \tilde{\xi}_t v + \tilde{\xi} g- \beta(a,x)  \xi  e^{s\widehat{\Phi}(0)}v(t,0,x)\right)^2 \,dxdadt \notag\\ 
& \leq C \bigg( \int_{Q}  (\xi_t)^2  e^{2s\widehat{\Phi}(0)} v^2  \,dxda dt
 + \int_0^{T^*}\int_0^A\int_0^1   \xi^2 e^{2s\widehat{\Phi}(0)} g^2 \,dxdadt\\ &+ \|\beta\|^2_{\infty}\int_0^{T^*}\int_0^A\int_0^1 \xi^2 e^{2s\hat\Phi(0)}v^2(t,0,x) dxdadt\bigg).
\end{aligned}
\end{equation}
Observing
that $ \supp \xi_t\subset [T_0, T^*]$, we deduce
\begin{equation}\label{right-energy-w-2}
\begin{aligned}
 \int_{Q}  (\xi_t)^2 e^{2s\widehat{\Phi}(0)}v^2  \,dxdadt &\le  C\int_{T_0}^{T^*} \int_0^A\int_0^1e^{2s\widehat{\Phi}(0)} v^2 dxdadt.
\end{aligned}
\end{equation}
Hence, by the estimates \eqref{energy-w}-\eqref{right-energy-w-2} and the fact that $ \supp \xi \subset [0, T^*]$, we find that
\begin{equation}\label{estim-obs-v-1}
\begin{aligned}
\|e^{s\widehat{\Phi}(0)}  v(0) &\|_{L^2(Q_{A,1})}^2  + \int_{0}^{T_0}\int_0^A\int_0^1 v^2 e^{2s\Phi}  \,dxdadt \\
& \leq C \Big(\int_{T_0}^{T^*}\int_0^A\int_0^1    v^2 e^{2s\widehat{\Phi}(0)}  \,dxdadt + \int_{0}^{T^*}\int_0^A\int_0^1e^{2s\hat \Phi(0)}  g^2 \,dxdadt\\
&+ \|\beta\|^2_{\infty} \int_{0}^{T^*}\int_0^A\int_0^1 e^{2s\widehat{\Phi}(0)}  v^2(t,0,x) dxdadt \Big).
\end{aligned}
\end{equation}
Now, we estimate the first term in the right-hand side of \eqref{estim-obs-v-1}. To this purpose,
fixed $\delta \in(0,A)$, observe that
\begin{equation}\label{inequality_delta_Phi}
- \varphi (t,a,x) \le -\Phi^*(T^*, \delta)
\end{equation}
for all $t \in [T_0, T^*]$, for all $a \in [\delta,A]$ and for all $x \in [0,1]$ and
\begin{equation}\label{=}
\Phi= \varphi \; \text{in} \; \left( T_0, T\right)\times Q_{A,1},
\end{equation}
since $\gamma(t,a) =\Theta(t,a)$ in $\left( T_0, T\right)\times(0,A)$,
Thus, by \eqref{inequality_delta_Phi} and \eqref{=}, we have
\begin{equation}\label{disv}
\begin{aligned}
&\int_{T_0}^{T^*}\int_0^A\int_0^1  v^2    e^{2s \widehat{\Phi}(0)}\,dxdadt  =\int_{T_0}^{T^*}\left(\int_0^\delta + \int_\delta^A\right)\int_0^1  v^2    e^{2s \widehat{\Phi}(0)}\,dxdadt\\
&= \int_{T_0}^{T^*}\int_0^\delta\int_0^1v^2e^{2s \widehat{\Phi}(0)}\, dxdadt+\int_{T_0}^{T^*}\int_\delta^A\int_0^1   v^2 e^{2s [\widehat{\Phi}(0)-\Phi]} e^{2s\Phi}\,dxdadt\\
&\leq \int_{T_0}^{T^*}\int_0^\delta\int_0^1 v^2 e^{2s\widehat{\Phi}(0)}
\, dxdadt +  e^{2s [\widehat{\Phi}(0)-\Phi^*(T^*,\delta)]} \int_{T_0}^{T^*}\int_\delta^A\int_0^1 v^2 e^{2s\Phi}\,dxdadt\\
&\leq \int_{T_0}^{T^*}\int_0^\delta\int_0^1 v^2 e^{2s\widehat{\Phi}(0)}
\, dxdadt + e^{2s [\widehat{\Phi}(0)-\Phi^*(T^*,\delta)]} \int_{T_0}^T\int_0^A\int_0^1  v^2 e^{2s\varphi}  \,dxdadt.
\end{aligned}
\end{equation}
Now, using the monotonicity of the function $x \mapsto \displaystyle \frac{x^2}{k(x)}$ and the Hardy-Poincar\'e inequality given in \cite[Proposition 2.1]{Alabau2006}, we have
\begin{align}
\int_{0}^{1} v^2 e^{2s\varphi}\,dx & \le \frac{1}{k(1)} \int_{0}^{1} \frac{k(x)}{x^{2}} (v e^{s\varphi})^2\,dx \notag \\ 
&\leq C \int_{0}^{1} k(x) (v e^{s\varphi})_x^2\,dx.
\end{align}
Hence, since $\varphi_x(t,a,x) = \Theta(t,a)\ds\frac{x}{k(x)}$, it follows that
\begin{align}
\int_0^1 v^2  e^{2s\varphi}\,dx & \leq C \int_{0}^{1}  \Big(k(x) v_x^2 + s^2 \Theta^2 \frac{x^2}{k(x)} v^2 \Big) e^{2s\varphi}\,dx.
\end{align}
By the previous inequality, \eqref{Carl_estimate-2} and \eqref{disv}, one has for $s$ large enough and for a positive constant $C$,
\begin{equation}\label{disv1}
\begin{aligned}
&\int_{T_0}^{T^*}\int_0^A\int_0^1  v^2  e^{2s \widehat{\Phi}(0)}\,dxdadt  \\
&
\leq \int_{T_0}^{T^*}\int_0^\delta\int_0^1 v^2 e^{2s\widehat{\Phi}(0)}
\, dxdadt + e^{2s [\widehat{\Phi}(0)-\Phi^*(T^*,\delta)]} \int_{T_0}^T\int_0^A\int_0^1  v^2 e^{2s\varphi}  \,dxdadt\\ &
\leq \int_{T_0}^{T^*}\int_0^\delta\int_0^1 v^2 e^{2s\widehat{\Phi}(0)}
 \, dxdadt \\
 &+Ce^{2s [\widehat{\Phi}(0)-\Phi^*(T^*,\delta)]}  \int_{0}^{T}\int_0^A\int_{0}^{1}  \Big(sk(x) v_x^2 + s^3 \Theta^2 \frac{x^2}{k(x)} v^2 \Big) e^{2s \varphi}\,dx\\
&\leq \int_{T_0}^{T^*}\int_0^\delta\int_0^1 v^2 e^{2s\widehat{\Phi}(0)}
 \, dxdadt +Ce^{-2s \Phi^*(T^*,\delta)}\int_{Q}e^{2s\hat \Phi(0)} g^{2}dxdadt \\
 &+Ce^{-2s\Phi^*(T^*,\delta)}\left( \int_{Q}e^{2s\hat \Phi(0)}  v^2(t,0,x) dxdadt +\int_{0}^{T} \int_0^A\int_{\omega} s^2 \Theta^2 v^2 e^{2s \eta}  dx dadt\right).
\end{aligned}
\end{equation}
Now, since the function $ s \rightarrow s^h e^{cs},$ with $h\geq 0$ and  $c<0$, is nonincreasing for larger values of $s$, then, from the fact that $\gamma \le \Theta $ in $(0, T) \times (0, A)$ we get that,  
\begin{equation}\label{dispesi}
 (s\Theta)^h e^{2s\eta}  \leq (s\gamma)^h e^{2s\sigma}, \quad \text{in} \; Q,
 \end{equation}
for $s$ large enough. Hence,
\[
\int_{0}^{T} \int_0^A\int_{\omega} s^2 \Theta^2 v^2 e^{2s\eta}  dx dadt\le\int_{0}^{T} \int_0^A\int_{\omega} s^2 \gamma^2 v^2 e^{2s\sigma}  dx dadt.
\]
Hence, \eqref{disv1} becomes
\begin{equation}\label{stima2new}
\begin{aligned}
&\int_{T_0}^{T^*}\int_0^A\int_0^1  v^2  e^{2s \widehat{\Phi}(0)}\,dxdadt \\
&\leq \int_{T_0}^{T^*}\int_0^\delta\int_0^1 v^2 e^{2s\widehat{\Phi}(0)}
 \, dxdadt +Ce^{-2s \Phi^*(T^*,\delta)} \int_{Q}e^{2s\hat \Phi(0)} g^{2}  dxdadt \\
 &+Ce^{-2s \Phi^*(T^*,\delta)}\left( \int_{Q} e^{2s\hat \Phi(0)} v^2(t,0,x)  dxdadt +\int_{0}^{T} \int_0^A\int_{\omega} s^2 \gamma^2 v^2 e^{2s\sigma}  dx dadt\right).
\end{aligned}
\end{equation}

This, together with the estimates \eqref{estim-obs-v-1} and \eqref{stima2new}  implies that
\begin{equation*}
\begin{aligned}
&\|e^{s\widehat{\Phi}(0)}  v(0) \|_{L^2(Q_{A,1})}^2  + \int_Q v^2 e^{2s\Phi}  \,dxdadt \\
&=\|e^{s\widehat{\Phi}(0)}v(0) \|_{L^2(Q_{A,1})}^2  + \int_{0}^{T_0}\int_0^A\int_0^1 v^2 e^{2s\Phi}  \,dxdadt +  \int_{T_0}^T\int_0^A\int_0^1 v^2e^{2s\varphi}   \,dxdadt \\
&\text{(proceeding as in \eqref{disv1})}\\
&  \leq C(s, T^*, \delta)\bigg(\int_{Q}e^{2s\hat \Phi(0)} g^{2}  dxdadt + \int_{0}^{T}\int_0^1 e^{2s\hat \Phi(0)} v^2(t,0,x) dxdt \\
&  + \int_0^T \int_0^A\int_{\omega} s^2 \gamma^2 v^2 e^{2s\sigma}  dx dadt + \int_{T_0}^{T^*}\int_0^{\delta}\int_0^1e^{2s\hat \Phi(0)}  v^2  dxdadt\bigg).
\end{aligned}
\end{equation*}

Thus the thesis follows.
\end{proof}
From now on we make the following hypothesis
\begin{assumptions}\label{conditionbeta} There exists  
$\bar a < A$
 such that 
\begin{equation}\label{conditionbeta1}
\beta(a, x)=0 \;  \text{for all $(a, x) \in [0, \bar a]\times [0,1]$}.
\end{equation}
\end{assumptions} 
Thanks to this assumption, we can prove the next estimate.
 \begin{proposition}\label{Prop_v(0)} Assume Hypotheses $\ref{ratesAss}$, $\ref{BAss01}$, $\ref{conditionbeta}$, $\bar a \le T$ and fix $T_0 \in [T-\bar a, T)$. Then, for all $\zeta \in (0,A)$ there exist
 two positive constants $C=C(s, \zeta, T_0, \bar a)$ and $s_0$ such that every solution $v \in \mathcal Z_T$ of  
\eqref{sys-adj-nonhom}
satisfies
\[
\begin{aligned}
&\int_{0}^{T}\int_0^1 v^2(t,0,x) dxdt \le C\int_{T-\bar a}^T\int_0^1 \int_0^{T-t}g^2(t+\tau, \tau,x) d\tau dxdt\\
& + C\Big(\int_Q g^2dxdadt+ \int_{{T-\bar a}}^{T} \int_0^A\int_{\omega} s^2 \gamma^2 v^2 e^{2s\sigma}  dx dadt \Big)\\
&+C\Big( \int_0^{\bar a}\int_0^1 v_T^2dxda+   \int_{T_0}^{T}\int_0^\zeta \int_0^1 v^2 dxdadt \Big).
\end{aligned}
\]
for all $s \geq s_0$. 
\end{proposition}
\begin{proof} Let $(S(t))_{t \ge0}$ be the semigroup generated by the operator $\mathcal A_0 -\mu Id$ for all $u \in D(\mathcal A_0)$  ($Id$ is the identity operator). Observe that, since the semigroup generated by $\mathcal A_0$ is a contraction semigroup and thanks to the hypothesis on $\mu$, we have that also the semigroup generated by $\mathcal A_0 -\mu Id$ is bounded. Now, set $\tilde T:= T- \bar a$ and
\begin{equation}\label{Gamma}
\Gamma:= \min \{\bar a, \Gamma_{A,T}\},
\end{equation}
where
$\Gamma_{A,T}:= A -a +t-\tilde T$. Then, as in \cite{fJMPA} and using the method of characteristic lines, one can prove
the following implicit formula for $v$ solution of \eqref{sys-adj-nonhom1}: here it should be \eqref{sys-adj-nonhom}

 if $t \ge  \tilde T + a$
\begin{equation}\label{implicitformula}
v(t,a, \cdot)=S(T-t)v_T(T+a-t, \cdot)- \int_a^{T-t+a} S(s-a) g(t-a+s,s)ds;
\end{equation}
if $t <  \tilde T + a$ and $\Gamma\!= \!\bar a$
\begin{equation}\label{implicitformula1}
v(t,a, \cdot)=
S(T-t) v_T(T+a-t, \cdot) +\int_a^{T+a-t}S(s-a)\left[\beta(s, \cdot)v(s+t-a, 0, \cdot) - g(s+t-a,s)\right] ds;
\end{equation}
finally if $t <  \tilde T + a$ and $\Gamma\!= \!\Gamma_{A,T}$
\begin{equation}\label{implicitformula2}
v(t,a, \cdot) = \int_a^AS(s-a)\left[\beta(s, \cdot)v(s+t-a, 0, \cdot) - g(s+t-a,s)\right]  ds.
\end{equation}
In particular, it results
\begin{equation}\label{v(0)}
v(t,0, \cdot):= S(T-t) v_T(T-t, \cdot) -  \int_0^{T-t} S(s) g(t+s,s)ds,
\end{equation}
if $t \ge T-\bar a$.
Obviously,
\begin{equation}\label{eq1}
\int_0^T \int_0^1 v^2(t,0,x) dxdt =\left( \int_0^{\tilde T} + \int_{\tilde T}^T\right) \int_0^1v^2(t,0,x)  dxdt.
\end{equation}
 Now, we estimate $\int_{\tilde T}^T \int_0^1v^2(t,0,x) dxdt$. By \eqref{v(0)}, we have
\begin{equation}\label{eq2}
\begin{aligned}
&\int_{\tilde T}^T \int_0^1 v^2(t,0,x) dxdt\le C \int_0^{\bar a}\int_0^1v_T^2(a,x)dxda+  C \int_{\tilde T}^T\int_0^1\left(\int_0^{T-t} g^2(t+\tau, \tau,x)  d\tau\right)dxdt.
\end{aligned}
\end{equation}
In order to estimate $\int_0^{\tilde T} \int_0^1v^2(t,0,x) dx dt$, we  define, for $\varsigma >0$, the function $w= e^{\varsigma a }v$.
Then $w$ satisfies
\begin{equation}\label{h=0'}
\begin{cases}
\ds \frac{\partial w}{\partial t} + \frac{\partial w}{\partial a}
+(k(x)w_{x})_x-(\mu(t, a, x)+ \varsigma) w =-\beta(a,x)w(t,0,x)+ e^{\varsigma a }g ,& (t,x,a) \in   Q,
\\[5pt]
w(t,a,0)=w(t,a,1) =0, &(t,a) \in  Q_{T,A},\\
w(T,a,x) =e^{\varsigma a} v_T(a,x), &(a,x) \in Q_{A,1}, \\
w(t,A,x)=0, & (t,x) \in  Q_{T,1},\\
w(t,0,x)=v(t,0,x), & (t,x) \in  Q_{T,1}.
\end{cases}
\end{equation}
Hence, multiplying the equation of \eqref{h=0'} by $-\ds w$ and integrating by parts on $\tilde Q:= (0,\tilde T) \times(0,A) \times(0,1)$, it results
\[
\begin{aligned}
&-\frac{1}{2}\int_{Q_{A,1}} w^2(\tilde T,a,x) dxda + \frac{1}{2} \int_0^{ \tilde T}\int_0^1 w^2(\tau,0,x)dx d\tau\\
&+ \varsigma \int_{\tilde Q}w^2(\tau,a,x) dxdad\tau \le \int_{\tilde Q}\beta w(\tau,0,x)wdxdad\tau - \int_{\tilde Q} e^{\varsigma a }gwdxdad\tau
\\
& \le \|\beta\|_{L^\infty(Q)}\frac{1}{\epsilon}\int_{ \tilde Q}w^2 dxdad\tau+ \epsilon A\|\beta\|_{L^\infty(Q)} \int_0^{\tilde T}\int_0^1 w^2(\tau,0,x)dxd\tau
\\
&+\epsilon\int_{ \tilde Q} g^2dxdad\tau+ \frac{ e^{2\varsigma A}}{\epsilon} \int_{\tilde Q} w^2dxdad\tau
\end{aligned}
\]
for $\epsilon >0$. Hence,
\[
\begin{aligned}
& \frac{1}{2} \int_0^{\tilde T}\int_0^1 w^2(\tau,0,x) dx d\tau+ \varsigma \int_{\tilde Q}w^2(\tau,a,x) dxdad\tau \\
&\le \frac{\|\beta\|_{L^\infty(Q)}+e^{2\varsigma A}}{\epsilon}\int_{\tilde Q}w^2dxdad\tau+ \epsilon A\|\beta\|_{L^\infty(Q)} \int_0^{\tilde T}\int_0^1 w^2(\tau,0,x)dxd\tau\\
&+\frac{1}{2}\int_{Q_{A,1}} w^2(\tilde T,a,x) dxda
+  \epsilon\int_{\tilde Q} g^2dxdad\tau.
\end{aligned}
\]
Choosing $\ds\epsilon= \frac{1}{4\|\beta\|_{L^\infty(Q)}A}$,  $\ds \varsigma =\frac{ \|\beta\|_{L^\infty(Q)}+e^{2\varsigma A}}{\epsilon}$  and recalling that $w =e^{\varsigma a}v$,  we have
\begin{equation}\label{3.56}
\begin{aligned}
 \frac{1}{4} \int_0^{ \tilde T}\int_0^1 v^2(\tau,0,x) dx d\tau &\le C(A, \|\beta\|_{L^\infty(Q)}) \Big( \int_{Q_{A,1}}  v^2(\tilde T,a,x) dxda
+ \int_{\tilde Q} g^2dxdad\tau \Big).
\end{aligned}
\end{equation}
Now, we will estimate $\int_{Q_{A,1}}  v^2(\tilde T,a,x) dxda$. To this purpose, define, for $\nu >0$, the function $z= e^{\nu t }v$.
Then $z$ satisfies
\begin{equation}\label{h=0''}
\begin{cases}
\ds \frac{\partial z}{\partial t} + \frac{\partial z}{\partial a}
+(k(x)z_{x})_x-(\mu(t, a, x)+ \nu) z =-\beta(a,x)z(t,0,x)+ e^{\nu t} g,& (t,x,a) \in  \hat Q,
\\[5pt]
z(t,a,0)=z(t,a,1) =0, &(t,a) \in \hat Q_{T,A},\\
  z(T,a,x) = e^{\nu T}v_T(a,x), &(a,x) \in Q_{A,1}, \\
  z(t,A,x)=0, & (t,x) \in \hat Q_{T,1},
\end{cases}
\end{equation}
where $\hat Q:= (\tilde T, T) \times Q_{A,1}$, $\hat Q_{T,A}:= (\tilde T, T) \times (0,A)$ and $\hat Q_{T,1}:= (\tilde T,T)\times (0,1)$.
Multiplying the equation of \eqref{h=0''} by $-\ds z$ and integrating by parts on $Q_t:= (\tilde T,t) \times(0,A) \times(0,1)$, it results
\begin{equation}\label{3.56new}
\begin{aligned}
&-\frac{1}{2}\int_{Q_{A,1}} z^2(t,a,x) dxda + \frac{e^{\nu  \tilde T}}{2} \int_{Q_{A,1}}  v^2(\tilde T,a,x) dxda + \frac{1}{2} \int_{\tilde T}^t\int_0^1  z^2(\tau,0,x) dx d\tau\\
&+ \nu \int_{Q_t} z^2(\tau,a,x) dxdad\tau \le \int_{Q_t} \beta z(\tau,0,x)zdxdad\tau - \int_{Q_t}e^{\nu \tau } g zdxdad\tau 
\\
& \le\frac{ \|\beta\|_{L^\infty(Q)}+ e^{2\nu T}}{\epsilon}\int_{Q_t}z^2dxdad\tau+ \epsilon A\|\beta\|_{L^\infty(Q)} \int_{\tilde T}^t\int_0^1 z^2(\tau,0,x)dxd\tau + \epsilon\int_{Q_t}g^2dxdad\tau,
\end{aligned}
\end{equation}
for $\epsilon >0$. Choosing $\ds\epsilon= \frac{1}{2\|\beta\|_{L^\infty(Q)}A}$ and $\ds\nu=\frac{ \|\beta\|_{L^\infty(Q)}+ e^{2 \nu T}}{\epsilon}$, 
 we have
\begin{equation}\label{eq10}
\begin{aligned}
 &\int_{Q_{A,1}}  v^2(\tilde T,a,x) dxda  \le C\int_{Q_{A,1}}z^2(t,a,x) dxda  + \epsilon\int_{Q_t}g^2dxdad\tau  \\
 &\le C  \Big( \int_{Q_{A,1}}v^2(t,a,x) dxda+ \int_{Q_t}g^2dxdad\tau \Big).
\end{aligned}
\end{equation}
Now, take $\zeta \in (0, A)$ and fix $T^*\in (T_0,T)$. Then, proceeding as in \cite{fJMPA} or in \cite{fAnona} and integrating the previous inequality over $\ds \left[{T_0}, {T^*} \right]$, one has
\begin{equation}\label{t=0_new}
\begin{aligned}
 \int_{Q_{A,1}} v^2(\tilde T,a,x) dxda  &\le C\int_{{T_0}}^{{T^*}} \int_{Q_{A,1}} v^2(t,a,x) dxdadt+  \epsilon C\int_{T_0}^{{T^*}} \int_{Q_{A,1}} g^2dxdadt\\
 &=  C \int_{{T_0}}^{{T^*}} \left(\int_0^\zeta + \int_\zeta^A \right)\int_0^1 v^2(t,a,x) dxdadt+  C\int_Q g^2dxdadt.
\end{aligned}
\end{equation}
Now, consider the term $\ds\int_{{T_0}}^{{T^*}}  \int_\zeta^A\int_0^1 v^2(t,a,x) dxdadt$. Proceeding as in \eqref{disv} - \eqref{disv1}, applying Carleman inequality \eqref{Carl_estimate-2} and \eqref{dispesi}, one has for $s$ large enough and for a positive constant $C$
\begin{equation}\label{estimav}
\begin{aligned}
&\int_{{T_0}}^{{T^*}}  \int_\zeta^A\int_0^1v^2(t,a,x)  dxdadt  \le e^{-2s\Phi^*(T^*, \zeta)} \int_{{T_0}}^{{T^*}}  \int_\zeta^A\int_0^1 v^2(t,a,x)e^{2s\tilde \varphi} dxdadt\\
&\leq  e^{-2s\Phi^*(T^*, \zeta)} \int_{{T-\bar a}}^{T} \int_\zeta^A\int_0^1\left(\tilde \Theta k v_x^2+ \tilde\Theta^3\frac{x^2}{k(x)}v^2 \right) e^{2s\tilde \varphi}  dxdadt\\
&\le  Ce^{-2s\Phi^*(T^*, \zeta)}\left(\int_{\hat Q}g^{2} e^{2s\tilde \eta} dxdadt +\int_{{T-\bar a}}^{T} \int_0^A\int_{\omega} s^2 \tilde \Theta^2 v^2 e^{2s\tilde \eta}  dx dadt\right) \\
 &+Ce^{-2s\Phi^*(T^*, \zeta)}\int_{\hat Q} v^2(t,0,x) e^{2s\tilde \eta} dxdadt\\
 &\le  Ce^{-2s\Phi^*(T^*, \zeta)} \Big(\int_{\hat Q}g^{2}  dxdadt  + \int_{{T-\bar a}}^{T} \int_0^A\int_{\omega} s^2 \gamma^2 v^2 e^{2s\sigma}\Big)  dx dadt\\
&+ Ce^{-2s\Phi^*(T^*, \zeta)} \int_{{T-\bar a}}^{T} \int_0^1 v^2(t,0,x) dxdt,
\end{aligned}
\end{equation}
where $\hat Q$ is defined as before, 
$\tilde \varphi (t,a,x)= \tilde \Theta(t,a) \psi(x)$, $\tilde \eta (t,a,x)= \tilde \Theta(t,a) \Psi(x)$ and $\tilde \Theta(t,a):=\ds\frac{1}{[(t-\tilde T)(T-t)]^4a^4}$.
In the last inequality we have used that
$$  \tilde\eta \leq \eta \leq \sigma \quad \text{in} \, Q .$$
Now, since the semigroup $(S(t))_{t \ge0}$ is bounded and $T-\bar a< T_0 <t$, one has, as in \eqref{eq2}:
\begin{equation}\label{t=02new}
\begin{aligned}
&\int_{{T-\bar a}}^{T} \int_0^1 v^2(t,0,x)dxdt\le C\int_0^{\bar a}\int_0^1v_T^2 dxda+  C \int_{\tilde T}^T\int_0^1\int_0^{T-t}g^2(t+\tau, \tau,x) d\tau dxdt.
\end{aligned}
\end{equation}
Hence, from \eqref{estimav} and \eqref{t=02new}, we get
\begin{equation}\label{eq3}
\begin{aligned}
\int_{{T_0}}^{{T^*}}  \int_\zeta^A\int_0^1 v^2  dxdadt &\leq C e^{-2s\Phi^*(T^*, \zeta)}\left( \int_{\hat Q}g^{2}  dxdadt+ \int_{T-\bar a}^T\int_0^1\int_0^{T-t}g^2(t+\tau, \tau,x) d\tau dxdt\right)\\
&+ C e^{-2s\Phi^*(T^*, \zeta)}  \Big( \int_{{T-\bar a}}^{T}  \int_0^A\int_{\omega} s^2 \gamma^2 v^2 e^{2s\sigma}  dx dadt+\int_0^{\bar a}\int_0^1 v_T^2 dxda\Big).
\end{aligned}
\end{equation}
Adding \eqref{t=0_new} and \eqref{eq3}, 
we infer that
\begin{equation}\label{tondo}
\begin{aligned}
\int_{Q_{A,1}}v^2(\tilde T,a,x) dxda &\le C\left( \int_{T_0}^{T^*}\int_0^\zeta \int_0^1 v^2 dxdadt  +  \int_Q g^2 dxdadt\right) \\
&+ C e^{-2s\Phi^*(T^*, \zeta)}  \Big( \int_{{T-\bar a}}^{T}\int_0^1 \int_0^{T-t}g^2(t+\tau, \tau,x) d\tau dxdt\\
&+ \int_0^{\bar a}\int_0^1 v_T^2 dxda +\int_{{T-\bar a}}^{T} \int_0^A\int_{\omega} s^2 \gamma^2 v^2  e^{2s\sigma}   dx dadt\Big).
\end{aligned}
\end{equation}
Thus, by \eqref{3.56} and \eqref{tondo}
\[
\begin{aligned}
&\int_0^{\tilde T} \int_0^1v^2(t,0,x) dxdt \le
 C\left(\int_{Q_{A,1}}  v^2(\tilde T,a,x) dxda 
+  \int_{ Q} g^2 dxdadt\right)\\
&\le C \left( \int_{T_0}^{T^*}\int_0^\zeta \int_0^1 v^2  dxdadt  +\int_Q g^2dxdadt\right)\\
&+ C e^{-2s\Phi^*(T^*, \zeta)}  \Big(\int_{\tilde T}^T\int_0^1 \int_0^{T-t}g^2(t+\tau, \tau,x) d\tau dxdt\\
&+ \int_0^{\bar a}\int_0^1 v_T^2 dxda + \int_{{T-\bar a}}^{{T}} \int_0^A\int_{\omega} s^2 \gamma^2 v^2 e^{2s\sigma}  dx dadt\Big).
\end{aligned}
\]
By the previous inequality, \eqref{eq1} and \eqref{eq2}, the thesis follows.
\end{proof}

As a consequence of Propositions \ref{Prop_modif_Carl_estimate_1} and \ref{Prop_v(0)}, taking $ T_0 \in \left[\max\left\{T-\bar a, \ds \frac{T}{2}\right\}, T \right)$ $\Big($ observe that $\max\left\{T-\bar a, \ds \frac{T}{2}\right\}= T-\bar a$ if $\bar a < \ds \frac{T}{2}$$\Big)$ and $\zeta=\delta$, we have the next result. 

\begin{proposition}\label{Lemma_modif_Carl_estimate_1} Assume Hypotheses $\ref{ratesAss}$, $\ref{BAss01}$, $\ref{conditionbeta}$,  $\bar a \le T$ and fix $T_0 \in \left[\max\left\{T-\bar a, \ds \frac{T}{2}\right\}, T\right)$ with $T_0 < T^*$. Then, for all $\zeta \in (0,A)$,
there exist two positive constants $C = C(s, T_0, T^*, T,\bar a, \zeta) $ and $s_0$ such that every solution $v \in \mathcal Z_T$ of \eqref{sys-adj-nonhom} satisfies
\begin{equation}\label{modif_Carl_estimate_2}
\begin{aligned}
&\| e^{s\widehat{\Phi}(0)} v(0) \|_{L^2(Q_{A,1})}^2  + \int_Qe^{2s \Phi} v^2   \,dxdadt \\
&  \leq C\bigg(\int_Q  e^{2s\hat \Phi(0)} g^2dxdadt +\int_{T-\bar a}^T\int_0^1\int_0^{T-t}e^{2s\hat \Phi(0)} g^2(t+\tau, \tau,x) d\tau dxdt+ \int_0^{\bar a}\int_0^1 e^{2s\hat \Phi(0)} v_T^2 dxda \\
& + \int_0^T \int_0^A\int_{\omega} s^2 \gamma^2 v^2e^{2s\sigma}   dx dadt + \int_{T_0}^{T}\int_0^{\zeta}\int_0^1e^{2s\hat \Phi(0)}  v^2 dxdadt\bigg)
\end{aligned}
\end{equation} 
for all $s \geq s_0$.
\end{proposition}

As a consequence of the previous result, we have:

 \begin{proposition}\label{Prop_modif_Carl_estimate_new} Assume Hypotheses $\ref{ratesAss}$, $\ref{BAss01}$, $\ref{conditionbeta}$ and $\bar a \le T$. Then, there exists two positive constants $C=C(s, T,\bar a) $ and $s_0$ such that every solution $v \in \mathcal Z_T$ of  
\eqref{sys-adj-nonhom}
satisfies
\begin{equation}\label{modif_Carl_estimate_new}
\begin{aligned}
&\| e^{s\widehat{\Phi}(0)} v(0) \|_{L^2(Q_{A,1})}^2  + \int_Q e^{2s \Phi} v^2   \,dxdadt\\ 
&   \leq C\bigg(\int_Q e^{2s\hat \Phi(0)} g^2 dxdadt+ \int_0^{\bar a}\int_0^1e^{2s\hat \Phi(0)}  v_T^2dxda  + \int_0^T \int_0^A\int_{\omega} s^2 \gamma^2 v^2 e^{2s\sigma}  dx dadt \bigg)
\end{aligned}
\end{equation}
for all $s \geq s_0$. 
\end{proposition}

\begin{proof} 
The thesis follows by Proposition \ref{Lemma_modif_Carl_estimate_1} taking  
$T_0 = \ds T-\frac{\bar a}{2}$,  and $\zeta \le \frac{\bar a}{2}$; indeed, it is sufficient to estimate
\[
 \int_{T-\frac{\bar a}{2}}^{T}\int_0^{\zeta }\int_0^1 v^2 dxdadt.
\]
In particular, we will prove 
 that there exists $C>0$ such that
\begin{equation}\label{t=06}
\begin{aligned}
&\ds \int_{{T - \frac{\bar a}{2}}}^{ T}  \int_0^{\zeta }\int_0^1 v^2(t,a,x) dxdadt  \le  C \int_0^{\bar a} \int_0^1   v_T^2(a,x)dxda  +  C \int_{{T - \frac{\bar a}{2}}}^{T}   \int_0^{\bar a} \int_0^1 g^2(t, \tau,x) da d\tau  dt.
\end{aligned}
\end{equation}
In order to prove \eqref{t=06}, we use \eqref{implicitformula}; indeed
\[
t \ge T- \frac{ \bar a}{2}= T- \bar a + \frac{\bar a}{2}\ge T- \bar a + \zeta  \ge T- \bar a + a = \tilde T +a.
\]
Taking into account these considerations, using the assumption on $\beta$, the boundedness of $(S(t))_{t \ge0}$ and \eqref{inequality_delta_Phi}, one has:
\begin{equation*}
\begin{aligned}
&\ds \int_{{T - \frac{\bar a}{2}}}^{T}  \int_0^{\zeta }\int_0^1v^2(t,a,x)  dxdadt\le C  \int_{ T - \frac{\bar a}{2}}^{T}  \int_0^{\zeta }\int_0^1 v_T^2(T+a-t,x)dxdadt\\
&+  C \int_{{T - \frac{\bar a}{2}}}^{T}  \int_0^{\zeta }\int_0^1  \int_a^{T-t+a} g^2(t-a+\tau,\tau,x)d\tau dxdadt
\\&\le C  \int_{0}^{{\frac{\bar a}{2}}}  \int_0^{\zeta }\int_0^1   v_T^2(a+z,x) dxdadz  +  C \int_{{T - \frac{\bar a}{2}}}^{T}  \int_0^1  \int_0^{\bar a}g^2(z, \tau,x) d\tau dx dz\\
& \le  C \int_0^{\bar a} \int_0^1   v_T^2(a,x)dxda  +  C \int_{{T - \frac{\bar a}{2}}}^{T}  \int_0^1  \int_0^{\bar a}g^2(z, \tau,x) d\tau dx dz.
\end{aligned}
\end{equation*}
Hence, we obtain \eqref{t=06}.
Analogously, one can prove that
\begin{equation}\label{bo}
\int_{T-\bar a}^T\int_0^1 \int_0^{T-t}g^2(t+\tau, \tau,x) d\tau dxdt\le \int_{T-\bar a}^T\int_0^{\bar a} \int_0^1 g^2(z,\tau,x)dxd\tau dz.
\end{equation}
 and the thesis follows multiplying \eqref{t=06} and \eqref{bo} by the term $e^{2s\hat \Phi(0)} $.
\end{proof}
\begin{remark}\label{rem}
Observe that Proposition \ref{Prop_modif_Carl_estimate_new} can be applied in particular
to the following system:
\begin{equation}\label{sys-adj-nonhom1}
\begin{cases}
\ds \frac{\partial v}{\partial t} + \frac{\partial v}{\partial a}
+(k(x)v_{x})_x-\mu(t, a, x)v =- \beta(a,x) v(t,0,x) +g ,& (t,a,x) \in  Q,
\\[5pt]
v(t,a,0)=v(t,a,1) =0, &(t,a) \in Q_{T,A},\\
  v(t,A,x)=0, & (t,x) \in Q_{T,1},\\
  v(T,a,x) = v_T(a,x) \in D(\mathcal A), &(a,x) \in Q_{A,1},
\end{cases}
\end{equation}
where  
\begin{equation}\label{stimag}
g(t,a,x)=- \int_t^Tb(\tau,t,a,x) w(\tau,a,x) d\tau,
\end{equation}
with $w \in L^2(Q)$.
\end{remark}
\subsection{The case $k(1)=0$}

In this subsection we will consider the case when $k$ degenerates only at $x=1$ and as before assume that $b$, $\mu$ and $\beta$ satisfy \eqref{3}. Therefore, on $k$ we make additional assumptions:
\begin{assumptions}\label{BAss01'} The function
$k\in C^0[0,1]\bigcap C^1[0,1)$  is such that $k(1)=0$, $k>0$ on
$[0,1)$ and there exists $M_2 \in (0,2)$ such that
 $(x-1)k'(x) \le M_2 k(x)$ for a.e. $x\in [0,1]$ for all $x \in [0,1]$. Moreover, if $M_2\ge 1$ one has to require that there exists $\theta \in (0,M_2]$, such that the function $x \ds \mapsto \frac{k(x)}{|1-x|^\theta}$ is nonincreasing near $1$.
\end{assumptions}
Now, we consider, as in \cite{fAnona}, the following weight functions 
\begin{equation}\label{varphi2}
\bar\varphi(t,a,x):=\Theta(t,a)\bar\psi(x),
\end{equation}
where $\Theta$ is defined as in \eqref{Theta},
\begin{equation}\label{psi2} 
\bar\psi(x):=\bar p(x) - 2 \|\bar p\|_{L^\infty(0,1)}
\end{equation}
and
$
\displaystyle \bar p(x):=\int_{0}^x\frac{y-1}{k(y)}dy$. As before,
$\bar\varphi(t,a, x) <0$ for all $(t,a,x) \in Q$ and
$\bar  \varphi(t,a, x)  \rightarrow - \infty \, \text{ as } t \rightarrow
0^+, T^-$  or  $a \rightarrow
0^+$.  We also define
\begin{equation}
\bar \eta(t,a, x)=\Theta(t, a)\bar \Psi(x), \quad  \bar \Psi(x)=e^{\kappa\bar \rho(x)}-e^{2 \kappa\|\bar\rho\|_{\infty}},
\end{equation}
$\ds
(t, a, x) \in Q, \quad \kappa>0 \quad\text {and} \quad\bar \rho(x):= \mathfrak{d} \int_{0}^{x} \frac{1}{k(t)} d t, \quad\text { where} \quad\mathfrak{d}=\left\|k^{\prime}\right\|_{L^{\infty}(0,1)}$.
By taking the parameter $\kappa$ as in \eqref{cond_kappap},
we have $\max\limits_{x \in [0,1]}\bar\psi(x) \le \min\limits_{x \in [0,1]}\bar\Psi(x)$,
thus
\begin{equation}\label{compar_varphi_eta2}
\bar\varphi(t,a,x) \leq \bar\eta(t,a,x), \;\, \text{for all} \;(t,a,x) \in Q.
\end{equation}
Moreover, in this case
$\bar\psi$ is nonincreasing and $\bar \Psi$ is increasing. 
As weight time functions blowing up only at the final time $t=T$, we consider
\begin{equation}\label{modif_weight_funct'}
\begin{aligned}
& \bar\Phi(t,a,x) :=  \gamma(t,a)\bar\psi(x), \qquad \bar\sigma(t,a,x): = \gamma(t,a)\bar\Psi(x),  \quad \forall\; (t,a, x) \in Q, \\
& \widehat{\bar\Phi}(t):= \min_{a\in(0,A]}\max\limits_{x\in[0,1]} \bar\Phi(t,a,x)= \gamma(t,A)\bar\psi(0),  \quad \forall \;t \in (0,T).\\
& \bar\Phi^*(t, a):=\min\limits_{x\in[0,1]}\bar \Phi(t,a,x)= \gamma(t, a)\bar\psi(1),  \quad \forall\; (t,a) \in Q_{T,A},\\
\end{aligned}
\end{equation}
Here $\gamma$ is defined as in \eqref{modif_weight_funct}.  
As in the previous subsection, one can prove the following results.
\begin{corollary}\label{Cor'}
Assume Hypotheses $\ref{ratesAss}$ and $\ref{BAss01'}$. Then,
there exist two strictly positive constants $C$ and $s_0$ such that every
solution $z \in \mathcal Z_T$ of \eqref{sys-adj-nonhom}
satisfies, for all $s \ge s_0$,
\begin{align}\label{Carl_estimate-2'}
\int_{Q} \Big(s \Theta k  v_x^2 & + s^3\Theta^3  \frac{(x-1)^2}{k} v^2\Big)e^{2s\bar\varphi} \, dxdadt \notag\\
&\leq C \Big(\int_{Q}g^{2} e^{2s\bar \eta} dxdadt + \int_Q v^2(t,0,x) e^{2s\bar\eta} dxdadt \Big)\notag\\
&+ C \int_0^T \int_0^A\int_{\omega} s^2 \Theta^2 v^2 e^{2s\bar\eta}  dx dadt.
\end{align}
\end{corollary}
Proceeding as in Proposition \ref{Prop_modif_Carl_estimate_1}, one can prove the following result. 
 \begin{proposition}\label{Prop_modif_Carl_estimate_1'} Assume Hypotheses $\ref{ratesAss}$, $\ref{BAss01'}$,  and let $T_0 \in [T/2, T)$ . Then, for all $\delta \in (0,A)$,there exist two positive constants $s_0$  and $C=C(s, T_0,  \delta)$ such that every solution $v \in \mathcal Z_T$ of \eqref{sys-adj-nonhom} satisfies
\begin{equation}\label{modif_Carl_estimate'}
\begin{aligned}
&\|e^{s\widehat{\bar\Phi}(0)}  v(0) \|_{L^2(Q_{A,1})}^2  + \int_Q v^2 e^{2s\bar\Phi}  \,dxdadt \\
&\leq C \bigg(\int_{Q}e^{2s\hat {\bar {\Phi}}(0)} g^{2} dxdadt + \int_0^T \int_0^1e^{2s\hat{\bar{ \Phi}}(0)}  v^2(t,0,x)  dxdt\\
&  + \int_0^T \int_0^A\int_{\omega} s^2 \gamma^2 v^2 e^{2s\bar\sigma}  dx dadt + \int_{T_0}^{T}\int_0^{\delta}\int_0^1e^{2s\hat{ \bar{\Phi}}(0)}  v^2 dxdadt\bigg)
\end{aligned}
\end{equation}
for all $s \geq s_0$. 
\end{proposition}
We underline that in the proof of the previous Proposition, in place of \eqref{inequality_delta_Phi}, one has to use the following estimate
\[- \bar\varphi (t,a,x) \le -\bar \Phi^*(\tau, \delta)\]
for  all $(t,a) \in (T_0,\tau) \times (\delta, A)$ with $\tau \in \left(T_0, T\right)$ and  $\delta\in (0,A)$.

As in the previous subsection, one can estimate the two integrals $\int_0^T \int_0^1 v^2(t,0,x)  dxdt$ and  $\int_{T_0}^{T}\int_0^{\delta}\int_0^1 v^2  dxdadt$ improving the previous result in the following sense:

 \begin{proposition}\label{Prop_modif_Carl_estimate_new'} Assume Hypotheses $\ref{ratesAss},$ $\ref{conditionbeta}$,  $\ref{BAss01'}$ and $\bar a \le T$. Then, there exist two positive constants $C=C(s, T, \bar a) $ and $s_0$ such that every solution $v \in \mathcal Z_T$ of  
\eqref{sys-adj-nonhom}
satisfies
\begin{equation}\label{modif_Carl_estimate_new_1}
\begin{aligned}
&\|e^{s\widehat{\bar\Phi}(0)}  v(0) \|_{L^2(Q_{A,1})}^2  + \int_Q v^2 e^{2s\bar\Phi}  \,dxdadt \\
&   \leq C\bigg(\int_Q e^{2s\hat{\bar{ \Phi}}(0)} g^2dxdadt+ \int_0^{\bar a}\int_0^1e^{2s\hat{\bar{ \Phi}}(0)}  v_T^2dxda  + \int_0^T \int_0^A\int_{\omega} s^2 \gamma^2 v^2 e^{2s\bar\sigma}  dx dadt \bigg)
\end{aligned}
\end{equation}
for all $s \geq s_0$. 
\end{proposition}
\section{Null controllability for the initial problem}\label{sect_null_control_nonhomog}
In this section we will prove the main result of this paper via a fixed point method. First of all we will prove a null controllability result for the following problem
\begin{equation}\label{problem001}\left\{
\begin{array}{lll}
\displaystyle y_t  + y_a- (k(x)  y_x )_x + \mu(t,a,x) y= \int\limits_0^t b(t,s,a,x)  w(s,a,x) \, ds + f(t,a,x) \chi_\omega & \text{in } Q, \\
y(t,a,1)= y(t,a,0)=0, & \text{in } Q_{T,A},  \\
y(0,a,x)= y_{0}(a,x),  &  \text{in }  Q_{A,1},\\
y(t,0,x)=\int_0^A \beta(a,x) y(t,a,x) da, & \text{in } Q_{T,1},\\
\end{array}
\right.\end{equation}
where $w \in L^2(Q)$ is fixed, then we extend this result to the initial problem.
For this reason we divide this section into two subsetcions. 
\subsection{Null controllability for \eqref{problem001}}
As a consequence of Proposition \ref{Prop_modif_Carl_estimate_new} or Proposition \ref{Prop_modif_Carl_estimate_new'}, we will obtain the null controllability  for  \eqref{sys-nonhom} with more regular initial data. Such result will play a fundamental role in obtaining null controllability for the initial problem \eqref{sys-memory1}. To this aim, following the arguments developed in \cite{Allal2020, AFS2020, FI1996,TG2016}, we denote by $s_0$ the parameter given in Proposition \ref{Prop_modif_Carl_estimate_new} or in Proposition \ref{Prop_modif_Carl_estimate_new'}. 
Now, we are ready to state the null controllability result for \eqref{sys-nonhom} under an additional assumption on the function $b$.
\begin{assumptions}\label{conditionb} Assume that there exists a positive constant $\gamma_0$ 
such that
\[
e^{2\frac{4^4 s \|p\|_{L^\infty(0,1)}}{T^4(T-t)^4a^4}}b \in L^{\infty}((0,T)\times Q),
\]
for all $s \ge \gamma_0$. Here $p$ is the same function of \eqref{psi}.

In this condition, $s$ should be fixed. Note also that, we only need to prove Theorem \ref{Thm_nul_ctrl_nonhomog_nondiv} for a fixed $s$ not for all $s \ge \max\{\gamma_0,s_0\}$.
\end{assumptions} 
\begin{theorem}\label{Thm_nul_ctrl_nonhomog_nondiv}
Assume  Hypotheses $\ref{ratesAss}$, $\ref{BAss01}$, $\ref{conditionbeta}$, $\ref{conditionb}$, $\bar a \le T$ and 
fix $w \in L^2(Q)$. Then, for any $y_0\in L^2(0,A; H^1_k(0,1))$, there exists a control function $f \in L^2(Q)$, such that the associated solution $y\in \mathcal Z_T$ of 
\eqref{problem001}
satisfies 
\[y(T,a,x)= 0, \quad \forall \; \;(a,x) \in  (\bar a, A) \times(0,1).\]

Moreover, there exists a positive constant $C$ such that the couple $(y, f)$ satisfies
\begin{equation}\label{mainestimate_nondiv}
\begin{aligned}
&\int_{Q}y^2e^{-2s\hat \Phi(0)}\,dx\,dt +
\int_0^T\int_0^A \int_\omega s^{-2} \gamma^{-2} f^{2} e^{-2s\sigma(t,x)}\,dx\,da dt + \int_0^{\bar a}\int_0^1 e^{-2s\hat \Phi(0)}  y^2(T,a,x)dxda\\
&\leq C   \Big(\|h e^{-s\Phi}\|^2_{L^2(Q)} +  \|y_0 e^{-s\hat{\Phi}(0)}\|^2_{L^2(Q_{A,1})}\Big),
\end{aligned}
\end{equation}
for all $s \ge \max\{\gamma_0,s_0\}$,
where
 $h$ is defined as
\begin{equation}\label{stimah}
h(t,a,x):=\int\limits_0^t b(t,s,a,x)  w(s,a,x)  ds.
\end{equation}
The same conclusions hold if Hypothesis $\ref{BAss01}$ is substituted by Hypothesis $\ref{BAss01'}$.
\end{theorem}

\begin{proof}
Following the ideas in \cite{cara,TG2016}, fixed $s \geq s_0$, let us consider the functional
\begin{equation}\label{extremal}
J(y,f)= \Big(\int_{Q} y^2e^{-2s\hat \Phi(0)}\,dx da dt +
\int_0^T\int_0^A\int_\omega  s^{-2}\gamma^{-2} f^{2} e^{-2s\sigma}dx da dt  + \int_0^{\bar a}\int_0^1 e^{-2s\hat \Phi(0)} y^2(T)dxda\Big),
\end{equation}
with $f\in L^2(Q)$, where $(y,f)$ satisfies
\begin{equation}\label{problem001_new}\left\{
\begin{array}{lll}
\displaystyle y_t  + y_a- (k(x)  y_x )_x + \mu(t,a,x) y= \int\limits_0^t b(t,s,a,x)  w(s,a,x) \, ds + f(t,a,x) \chi_\omega & \text{in } Q, \\
y(t,a,1)= y(t,a,0)=0, & \text{in } Q_{T,A},  \\
y(0,a,x)= y_{0}(a,x),  &  \text{in }  Q_{A,1},\\

y(t,0,x)=\int_0^A \beta(a,x) y(t,a,x) da, & \text{in } Q_{T,1}\\
y(T,a,x)=0 ,&  \text{in }  (\bar a, A)\times (0,1).
\end{array}
\right.\end{equation}
The functional
$J$ attains its minimizer at a unique point  denoted as $(\bar{y},\bar{f})$ (see \cite{Lions,Lionsb}).
Setting
$$
L_{k}y:= y_t  + y_a- (k(x)  y_x )_x + \mu(t,a,x) y \qquad \text{in}\quad Q
$$
and $L_{k}^{\star}$ the (formally) adjoint operator of $L_{k}$, we will first prove that there exists a dual variable $\bar z$ such that
\begin{equation}\label{step0}
\left\{
  \begin{array}{ll}
\bar{y}= e^{2s\hat \Phi(0)} L_k^{\star} \bar{z}, &\quad \text{in}\quad Q,\\
\bar{y}(T)=  e^{2s\hat \Phi(0)} \bar z(T), & \quad \text{in}\quad (0,\bar a)\times (0,1),\\
\bar{f}= - s^2 \gamma^2 e^{2s\sigma}\bar{z},  &\quad \text{in}\quad Q,\\
\bar{z}=0, &\quad \text{on}\quad (0,T)\times (0,A)\times \{0,1\}\cup (0,T)\times \{A\}\times (0,1).
  \end{array}
\right.
\end{equation}

First of all, let us start by introducing the following linear space
$$
\mathcal{P}_0=\big\{z \in C^{\infty}(\overline{Q}): \; z \text{ satisfies } \eqref{sys-adj-nonhom1} \big\}.
$$
Clearly, $\mathcal P_0$ is a vector space. Now, introduce on it the bilinear form $\Gamma$:
\begin{equation*}
\begin{aligned}
\Gamma(z_1,z_2)&= \int_{Q} e^{2s\hat \Phi(0)} L_{k}^{\star} z_1 L_{k}^{\star} z_2\,dxda dt
+\int_0^T\int_0^A \int_\omega  s^{2}\gamma^2 e^{2s\sigma } z_1 z_2\,dxdadt\\
&+\int_0^{\bar a}\int_0^1e^{2s\hat \Phi(0)}  z_1(T,a,x)z_2(T,a,x)dxda.
\end{aligned}
\end{equation*}
In particular
\[
\begin{aligned}
\Gamma(z,z)&= \int_{Q} e^{2s\hat \Phi(0)}( L_{k}^{\star} z)^2\,dxda dt
+\int_0^T\int_0^A \int_\omega  s^{2}\gamma^2 e^{2s\sigma} z^2\,dxdadt+\int_0^{\bar a}\int_0^1e^{2s\hat \Phi(0)} z^2(T,a,x)dxda.
\end{aligned}
\]
Observe that if $\bar y$, $\bar f$ defined in \eqref{step0} satisfy \eqref{problem001_new} then we must  have
\begin{equation}\label{step00}
\Gamma(\bar{z},z)=\int_{Q} hz\,dxdadt + \int_0^A\int_{0}^{1}y_0 z(0)\,dxda,\quad \forall\; z\in \mathcal{P}_0.
\end{equation}
The key idea in this proof is to show that there exists exactly one $\bar z$ satisfying \eqref{step00} in an
appropriate class. We will then define $\bar y$ and $\bar f$ using \eqref{step0} and we will check that the couple
$(\bar y, \bar f)$ satisfies the desired properties.

Observe that the modified Carleman inequality \eqref{modif_Carl_estimate_new} holds for all $z \in \mathcal{P}_{0}$. Consequently, 
\begin{equation}\label{step000}
\begin{aligned}
&\|e^{s\widehat{\Phi}(0)}  z(0) \|_{L^2(Q_{A,1})}^2  + \int_Q z^2 e^{2s\Phi}  \,dxdadt \\
&  \leq C \bigg(\int_Qe^{2s\hat \Phi(0)} g^2dxdadt+ \int_0^{\bar a}\int_0^1 e^{2s\hat \Phi(0)}z^2(T,a,x)   dxda+ \int_0^T \int_0^A\int_{\omega} s^2 \gamma^2 z^2 e^{2s\sigma}  dx dadt\bigg)\\
&= C \Gamma(z,z).
\end{aligned}
\end{equation}
In particular, $\Gamma(\cdot,\cdot)$ is a strictly positive and symmetric bilinear form, that is, $\Gamma(\cdot,\cdot)$ is a scalar product in $\mathcal{P}_{0}$.

Denote by $\mathcal{P}$ the Hilbert space which is the completion of $\mathcal{P}_{0}$ with respect to the norm associated
to $\Gamma(\cdot,\cdot)$ (which we denote by $\|\cdot\|_{\mathcal{P}}$). Let us now consider the linear form $l$, given by
\begin{equation*}
l(z)=\int_{Q} hz\,dxdadt +  \int_0^A\int_{0}^{1}y_0 z(0)\,dxda, \quad
\forall \; z\in \mathcal{P}.
\end{equation*}
Now, observe that $h e^{-s\Phi}\in L^2(Q)$.
Indeed , setting $C_0:=2\frac{4^4}{T^4}\|p\|_{L^\infty(0,1)}$, by Hypothesis \ref{conditionb} and the estimate $e^{-s\Phi}\leq e^{\frac{sC_0}{(T-t)^4a^4}}$, we get that

 \begin{equation}\label{dis}
\begin{aligned}
&\int_{Q}   e^{-2s\Phi}\Big(\int_{0}^{t} b(t,\tau,a,x)  w(\tau,a,x)  d\tau\Big)^2 \,dx da dt \le C_T \int_Q   e^{-2s\Phi}\int\limits_0^t b^2(t,\tau,a,x)  w^2(\tau,a,x)  d\tau dxdadt\\
&
\le C_T \int_Q\int_0^t e^{\frac{2sC_0}{(T-t)^4a^4}}b^2(t,\tau,a,x)  w^2(\tau,a,x)  d\tau dxdadt\le C_T \int_Q w^2 dxdad\tau\\
&
= C_T \, e^{2s\hat \Phi(0)}\int_{Q} e^{-2s\hat \Phi(0)}w^2(\tau, a ,x)dxda d\tau <+\infty.
\end{aligned}
\end{equation}
Thus, by the Cauchy-Schwarz inequality and in view of \eqref{step000}, we have that
\begin{align*}
|l(z)| &\leq \left\|h e^{-s\Phi} \right\|_{L^2(Q)}  \|z e^{s\Phi}\|_{L^2(Q)}
 + \|y_0e^{-s\hat{\Phi}(0)} \|_{L^2(Q_{A,1})} \|z(0) e^{s\hat{\Phi}(0)}\|_{L^2(Q_{A,1})} \\
&\leq C    \Big(\|h e^{-s\Phi}\|_{L^2(Q)} +  \|y_0 e^{-s\hat{\Phi}(0)}\|_{L^2(Q_{A,1})}\Big)\|z\|_{\mathcal{P}}.
\end{align*}
Hence $l$ is a linear continuous form on $\mathcal{P}$ and, in view of Lax-Milgram's Lemma, there
exists one and only one $\bar{z}\in \mathcal{P}$ satisfying
\begin{equation}\label{step001}
\Gamma(\bar{z},z)= l(z),\quad \forall \;z\in \mathcal{P},
\end{equation}
i.e.
\eqref{step00}.
Moreover, we have
\begin{equation}\label{step002}
\|\bar{z}\|_{\mathcal{P}}\leq C \Big(\|h e^{-s\Phi}\|_{L^2(Q)} +  \|y_0 e^{-s\hat{\Phi}(0)}\|_{L^2(Q_{A,1})}\Big).
\end{equation}

Now, define $\bar y$ and $\bar f$ as in \eqref{step0}; thus by \eqref{step002}, it is easy to check that
\[
\begin{aligned}
\|\bar z\|^2_{\mathcal P}&= \int_{Q} e^{2s\hat \Phi(0)}( L_{k}^{\star} \bar z)^2\,dxda dt
+\int_0^T\int_0^A \int_\omega  s^{2}\gamma^2 e^{2s\sigma}\bar  z^2\,dxdadt+\int_0^{\bar a}\int_0^1e^{2s\hat \Phi(0)}\bar z^2(T,a,x)dxda\\
&= \int_{Q} e^{-2s\hat \Phi(0)}\bar y^2\,dxda dt +
\int_0^T\int_0^A \int_\omega s^{-2} \gamma^{-2}\bar f^{2} e^{-2s\sigma} \,dx\,da dt +  \int_0^{\bar a}\int_0^1 e^{-2s\hat \Phi(0)}\bar y^2(T, a,x)dxda.
\end{aligned}
\]
Thus,
\begin{equation}\label{step004}
\begin{aligned}
&\int\!\!\!\!\!\int_{Q} e^{-2s\hat \Phi(0)}\bar y^2\,dx\,dt +
\int_0^T\int_0^A \int_\omega s^{-2} \gamma^{-2}\bar f^{2} e^{-2s\sigma} \,dx\,da dt + \int_0^{\bar a}\int_0^1 e^{-2s\hat \Phi(0)} \bar y^2(T,a,x)dxda\\
&\leq C   \Big(\|h e^{-s\Phi}\|^2_{L^2(Q)} +  \|y_0 e^{-s\hat{\Phi}(0)}\|^2_{L^2(Q_{A,1})}\Big),
\end{aligned}
\end{equation}
which implies \eqref{mainestimate_nondiv}.

It remains to check that $\bar{y}$ is the solution of \eqref{problem001_new} corresponding to $\bar{f}$.  First of all, it is immediate that $\bar{y}$ and
$\bar{f}$ belong to  $L^2(Q)$. Denote by $\tilde{y}$ the (weak) solution of \eqref{sys-nonhom}
associated to the control function $f=\bar{f}$,  where $h$ is defined as in \eqref{stimah} , then $\tilde{y}$ is also the unique solution of \eqref{sys-nonhom} defined by transposition. In other words, $\tilde{y}$ is the unique function in $L^2(Q)$ satisfying
\begin{equation}\label{trans}
\begin{aligned}
&\int_{Q} \tilde{y} g\,dx dadt + \int_0^{\bar a}\int_0^1 \tilde y(T) z(T) dxda\\
&= \int_0^T\int_0^A \int_\omega \bar{f} z\,dxda dt + \int_{Q} h z\,dxdadt + \int_0^A\int_{0}^{1}y_0 z(0)\,dx da,
\end{aligned}
\end{equation}
where $g$ is defined in \eqref{stimag} and  $z$ is the solution of
\[
\begin{cases}
\ds \frac{\partial z}{\partial t} + \frac{\partial z}{\partial a}
+(k(x)z_{x})_x-\mu(t, a, x)z + \beta(a,x) z(t,0,x) =g(t,a,x),& (t,a,x) \in  Q,
\\[5pt]
z(t,a,0)=z(t,a,1) =0, &(t,a) \in Q_{T,A},\\
z(t,A,x)=0, & (t,x) \in Q_{T,1},\\
  z(T,a,x) = 0, &(a,x) \in (0,\bar a) \times (0,1).
\end{cases}
\]
According to \eqref{step0} and \eqref{step00}, we see that $\bar{y}$ also satisfies \eqref{trans}. Therefore, $\bar{y}=\tilde{y}$. Consequently, the control $\bar{f}\in L^2(Q)$ drives the state $\bar{y}\in \mathcal Z_T$ exactly to zero at time $T$ for all $(a,x) \in (\bar a,A) \times (0,1)$.
\end{proof}
\subsection{Null controllability for \eqref{sys-memory1}}
As a consequence of Theorem \ref{Thm_nul_ctrl_nonhomog_nondiv} , we will obtain the null controllability  for the initial problem with more regular initial data. 
\begin{theorem}\label{NCinitial}
Assume  Hypotheses $\ref{ratesAss}$, $\ref{BAss01}$, $\ref{conditionbeta}$, $\ref{conditionb}$ 
 and $\bar a \le T$. Then, for any $y_0\in L^2(0,A; H^1_k(0,1))$, there exists a control function $f \in L^2(Q)$, such that the associated solution $y\in \mathcal Z_T$ of 
\eqref{sys-memory1}
satisfies 
\begin{equation}\label{ncresult}
y(T,a,x)= 0, \quad \forall \;(a,x) \in  (\bar a, A) \times(0,1).
\end{equation}
The same conclusion holds if Hypothesis $\ref{BAss01}$ is substitued by Hypothesis $\ref{BAss01'}$.
\end{theorem}

\begin{proof}
Let $\mathcal{A}(w)$ the family of all controls 
$f \in L^2(Q)$ which satisfy the estimate
\eqref{mainestimate_nondiv} and drive the solution of \eqref{problem001} to zero at time $T$ for every $(a,x) \in  (\bar a, A) \times(0,1)$. In view of
Theorem \ref{Thm_nul_ctrl_nonhomog_nondiv}, this set is not empty. Now,
fix $R>0$, let us now introduce, for every $w\in B_{R}:= \{ y\in \mathcal Z_T: \|ye^{-s\hat \Phi(0)}\|_{L^2(Q)}\le R\}$, the multivalued map
$$\Lambda: B_R \subset \mathcal Z_T\rightarrow 2^{\mathcal Z_T}$$
with
\begin{align*}
\Lambda(w)=\displaystyle\Big\{&y\in \mathcal Z_T:  (y,f) \text{ solves }\eqref{problem001} \; \text{with} \; f\in\mathcal{A}(w) \Big\}.
\end{align*}\;

Observe that if $y\in \Lambda (w)$, then $y(T, a,x)=0$ for all $(a,x) \in (\bar a, A) \times (0,1)$ and
$f$ satisfies
\[\int_0^T\int_0^A \int_\omega s^{-2} \gamma^{-2} f^{2} e^{-2s\sigma(t,x)}\,dx\,da dt 
\leq C\Big(R^2+ \int_0^A\int_{0}^{1}e^{-2s\hat{\Phi}(0)}y_0^2 \,dxda\Big)\]
thanks to Theorem \ref{Thm_nul_ctrl_nonhomog_nondiv}. Recall that $w\in B_R$.

To achieve our goal, it will suffice to show that $\Lambda$ possesses at least one fixed point.
To this purpose, we shall apply Kakutani's fixed point Theorem (see \cite[Theorem 2.3]{cara}).


First of all, it is readily seen that $\Lambda(w)$ is a nonempty, closed and convex subset of $L^2(Q)$
for every $w\in B_R$. Then, we prove that $\Lambda(B_R)\subset B_R$ with sufficiently large $R>0$.
By \eqref{mainestimate_nondiv} and \eqref{dis}, we have

\begin{equation}\label{starstima}
\begin{aligned}
&\int_{Q}y^2e^{-2s\hat \Phi(0)}\,dx\,dt +
\int_0^T\int_0^A \int_\omega s^{-2} \gamma^{-2} f^{2} e^{-2s\sigma(t,x)}\,dx\,da dt + \int_0^{\bar a}\int_0^1e^{-2s\hat \Phi(0)} y^2_T(a,x)dxda\\
&\leq C   \Big(\|h e^{-s\Phi}\|^2_{L^2(Q)} +  \|y_0 e^{-s\hat{\Phi}(0)}\|^2_{L^2(Q_{A,1})}\Big)\\
&\leq   C \Big( e^{2s\hat \Phi(0)}R^2+ \int_0^A\int_0^1 e^{-2s\hat{\Phi}(0)}y_0^2\,dxda\Big)\le R^2
\end{aligned}
\end{equation}
 for $R$ large enough, we obtain
\begin{equation*}
\int_{Q}e^{-2s\hat \Phi(0)}y^2\,dxda dt
\leq R^2.
\end{equation*}

It follows that $\Lambda(B_R) \subset B_R$.
Furthermore, let $\{w_n\}$ be a sequence of $B_R$. The regularity assumption on $y_0$
and Theorem \ref{theorem_existence}, imply that the associated solutions $\{y_n\}$ are bounded in  $\mathcal Z_T$.
Therefore, $\Lambda(B_R)$ is a relatively compact subset of $L^2(Q)$ by
the Aubin-Lions Theorem \cite{simon}.

In order to conclude, we have to prove that $\Lambda$ is upper-semicontinuous under the $L^2(Q)$-topology.
First, observe that for any $w\in B_R$, we have at least $f \in L^2(Q)$ such that the corresponding solution $y\in B_R$.
Hence, taking $\{w_n\}$ a sequence in $B_R$, we can find a sequence of controls $\{f_n\}$ such that the corresponding solutions
$\{f_n\}$ is in $L^2(Q)$. Thus, let $\{w_n\}$ be a sequence satisfying $w_n \rightarrow w$ in $B_R$ and
$y_n \in \Lambda(w_n)$ such that $y_n \rightarrow y$ in $L^2(Q)$. We must prove that
$y \in \Lambda(w)$.
For every $n$, we have a control $f_n \in L^2(Q)$ such that the system
\begin{equation}\label{sys_n}
\left\{
\begin{array}{lll}
\displaystyle y_{n,t}  + y_{n,a}- (k(x)  y_{n,x} )_x + \mu(t,a,x) y_n= \int\limits_0^t b(t,s,a,x)  w_n(s,a,x) \, ds + f_n(t,a,x) \chi_\omega & \text{in } Q, \\
y_n(t,a,1)= y_n(t,a,0)=0, & \text{in } Q_{T,A},  \\
y_n(0,a,x)= y_{0}(a,x),  &  \text{in }  Q_{A,1},\\
y_n(t,0,x)=\int_0^A \beta(a,x) y_n(t,a,x) da, & \text{in } Q_{T,1},\\
\end{array}
\right.
\end{equation}
has a least one solution $y_n\in L^2(Q)$ that satisfies
\begin{equation*}
y_n(T,a,x)=0 \qquad \forall \; (a,x) \in (\bar a ,A)\times (0,1).
\end{equation*}
From Theorem \ref{theorem_existence} and \eqref{starstima}, it follows (at least for a subsequence) that
\begin{align*}
 f_n \rightarrow f \quad &\text{weakly in} \; L^2(Q),\notag \\
 y_n \rightarrow y \quad &\text{weakly in} \; \mathcal Z_T,\\
 &\text{strongly in} \; C(0, T;  L^2(Q_{A,1})).
\end{align*}
Passing to the limit in \eqref{sys_n}, we obtain a control $f \in L^2(Q)$ such that the corresponding solution $y$ to \eqref{problem001}
satisfies \eqref{ncresult}. This shows that $y\in \Lambda(w)$ and, therefore, the map $\Lambda$ is upper-semicontinuous.

Hence, the multivalued map $\Lambda$ possesses at least one fixed point, i.e., there exists $y \in \mathcal Z_T$ such that $y\in \Lambda(y)$.
By the definition of $\Lambda$, this implies that there exists at least one pair $(y,f)$ satisfying the conditions of the theorem.
The uniqueness of $y$ follows by Theorem \ref{theorem_existence}.
Thus the proof is complete.

\end{proof}

Clearly, Theorem  \ref{NCinitial}  holds also in a general domain $(t^*,T) \times (0,1)$ with suitable changes. Thanks to this fact, the following null controllability result holds for memory system \eqref{sys-memory1}. 
\begin{theorem}\label{Thm_null_Control_memo_2}
Assume  Hypotheses $\ref{ratesAss}$, $\ref{BAss01}$, $\ref{conditionbeta}$, $\ref{conditionb}$ 
 and $\bar a \le T$. Then, for any $y_0\in L^2(Q_{A,1})$, there exists a control function $f \in L^2(Q)$, such that the associated solution $y\in \mathcal W_T$ of 
\eqref{sys-memory1}
satisfies  \eqref{ncresult}.
The same conclusion holds if Hypothesis $\ref{BAss01}$ is substitued by Hypothesis $\ref{BAss01'}$.
\end{theorem}
The proof is similar to the one of \cite[Theorem 5]{Allal2020} or of \cite[Theorem 3.2]{AFS2020}, so we omit it.

\subsection{Null controllability in the case  $k(0)=k(1)=0$}
In this subsection we will extend the null controllability result proved above to the degenerate parabolic equation 
\eqref{sys-memory1}
when $k$ vanishes at both extremities of the interval $(0, 1)$ and satisfies, as in \cite{Allal2020} in the case that $y$ is independent of $a$, one of the four following cases:
\begin{itemize}
\item weakly-weakly degenerate case (WWD): 
\begin{equation*}
\left\{
\begin{array}{lll}
k \in C([0,1]) \cap C^1((0,1)), \; k(0)=k(1)=0, \; k>0 \quad \text{in}\quad (0, 1), \\
\exists \, M_1 \in [0,1),\text{ s.t. }x k'(x) \leq M_1 k(x)\;\; \text{for a.e.} \, x \in [0,1],\\
\exists \; {M_2} \in [0,1),\text{ s.t. } (x-1) k'(x) \leq M_2 k(x)\;\;\text{for a.e.}\; x \in [0,1],
\end{array}
\right.
\end{equation*}
\item strongly-weakly degenerate case (SWD):
\begin{equation*}
\left\{
\begin{array}{lll}
k \in C([0,1]) \cap C^1([0,1)), \; k (0)=k (1)=0, \; k>0 \quad \text{in}\quad (0, 1), \\
\exists \,M_1\in [1,2), \text{ s.t. } x k'(x) \leq M_1 k (x) \;\;\text{for a.e.}\, x \in [0,1] \text{ and } \exists \;\theta \in (0,M_1]  \text{ s.t. }\\
\quad x \ds \mapsto \frac{k(x)}{x^\theta} \text{ in nondecreasing near } 0,\\
\exists \; {M_2} \in [0,1),  \text{ s.t. } M_2(x-1) k'(x) \leq M_2 k(x)\;\; \text{for a.e.} \; x \in [0,1],
\end{array}
\right.
\end{equation*}
\item weakly-strongly degenerate case (WSD): 
\begin{equation*}
\left\{
\begin{array}{lll}
k \in C([0,1]) \cap C^1((0,1]), \; k(0)=k(1)=0, \; k>0 \quad \text{in}\quad (0, 1), \\
\exists \, M_1 \in [0,1),\text{ s.t. }\quad x k'(x) \leq M_1 k(x)\; \;\text{for a.e.}\, x \in [0,1],\\
\exists \,M_2\in [1,2), \text{ s.t. } (x-1) k'(x) \leq M_2 k (x)\;\; \text{for a.e.}\, x \in [0,1] \text{ and } \exists \;\theta \in (0,M_2]  \text{ s.t. }\\
\quad x \ds \mapsto \frac{k(x)}{|1-x|^\theta} \text{ in nondecreasing near } 1,\\
\end{array}
\right.
\end{equation*}
\item strongly-strongly degenerate case (SSD): 
\begin{equation*}
\left\{
\begin{array}{lll}
k \in C^1([0,1]), \; k(0)=k(1)=0, \; a>0 \quad \text{in}\quad (0, 1), \\
\exists \,M_1\in [1,2), \text{ s.t. } x k'(x) \leq M_1 k (x) \;\; \text{for a.e.}\, x \in [0,1] \text{ and } \exists \;\theta \in (0,M_1]  \text{ s.t. }\\
\quad x \ds \mapsto \frac{k(x)}{x^\theta} \text{ in nondecreasing near } 0,\\
\exists \,M_2\in [1,2), \text{ s.t. } (x-1) k'(x) \leq M_2 k (x)\;\; \text{for a.e.} \, x \in [0,1] \text{ and } \exists \;\theta \in (0,M_2]  \text{ s.t. }\\
\quad x \ds \mapsto \frac{k(x)}{|1-x|^\theta} \text{ in nondecreasing near } 1,\\
\end{array}
\right.
\end{equation*}
\end{itemize}
A typical example is $ k(x) = x^{M_1} (1 - x)^{M_2}, \text{with} \; M_1, M_2 \in [0,2)$.

We remember that in all the previous cases \eqref{sys-memory1} is still  well-posed via Theorem \ref{prop-Well-posed_memory}.  As in \cite{Allal2020},  as a consequence of Theorem \ref{Thm_null_Control_memo_2}, one can prove the null controllability for the initial problem when $k$ vanishes at both extremities of the interval $(0, 1)$.
\begin{theorem}\label{Thm_null_Control_memo_4}
Assume that $k$ is $(WWD)$, $(SSD)$, $(WSD)$ or $(SWD)$,  Hypotheses $\ref{ratesAss}$, $\ref{conditionbeta}$, $\ref{conditionb}$ 
 and $\bar a \le T$. Then, for any $y_0\in L^2(Q_{A,1})$, there exists a control function $f \in L^2(Q)$, such that the associated solution $y\in \mathcal W_T$ of 
\eqref{sys-memory1}
satisfies 
\eqref{ncresult}.
\end{theorem}
\section{Appendix} \label{appendix}
\subsection{Proof of Theorem \ref{theorem_existence}}\label{proof_existence} Since $\mathcal A$ is the infinitesimal generator of a strongly continuous semigroup on $L^2(Q_{A,1})$, it follows that if $y_0 \in L^2(Q_{A,1})$, then \eqref{sys-nonhom} has a unique solution $y \in C([0,T]; L^2(Q_{A,1}))$. Moreover, if $y_0 \in L^2(0,A; H^1_k(0,1))$, then $y \in L^2(0, T; L^2(0,A; H^1_k(0,1)))\cap H^1(0,T; L^2(Q_{A,1}))$. 
Multiplying the equation of \eqref{sys-nonhom} by $y$ and  integrating over $(0,A) \times (0,1)$, we obtain
\[
\begin{aligned}
&\frac{1}{2}\frac{d}{dt}\|y(t)\|^2_{L^2(Q_{A,1})}+ \frac{1}{2}\int_0^1y^2(t,A,x)dx -\frac{1}{2}\int_0^1y^2(t,0,x)dx +\int_0^A\int_0^1ky_x^2dxda \\
&=-\int_0^A\int_0^1\mu y^2dxda + \int_0^A\int_\omega fy dx da + \int_0^A\int_0^1 hy dx da.
\end{aligned}
\]
Hence, using the initial condition $y(t,0,x) = \int_0^A \beta(a,x) y(t,a,x) da$, the assumptions on $\beta$ and $\mu$ and the Jensen's inequa\-li\-ty, one has
\begin{equation}\label{stima1}
\begin{aligned}
&\frac{1}{2}\frac{d}{dt}\|y(t)\|^2_{L^2(Q_{A,1})} + \frac{1}{2}\int_0^1y^2(t,A,x)dx  +\int_0^A\int_0^1ky_x^2dxda \\
&\le \frac{C}{2}\int_0^A\int_0^1y^2(t,a,x)dxda + \frac{1}{2} \int_0^A\int_0^1\chi_\omega f^2 dx da+ \frac{1}{2} \int_0^A\int_0^1 h^2 dx da,
\end{aligned}
\end{equation}
where $C$ is a positive constant.
Since $\int_0^1y^2(t,A,x)dx $ and $\int_0^A\int_0^1ky_x^2dxda$ are positive, we deduce
\[
\frac{d}{dt}\|y(t)\|^2_{L^2(Q_{A,1})} \le C\|y(t)\|^2_{L^2(Q_{A,1})} + \|\chi_\omega f(t)\|^2_{L^2(Q_{A,1})}+  \|h(t)\|^2_{L^2(Q_{A,1})}.
\]
Setting $F(t):= \|y(t)\|^2_{L^2(Q_{A,1})} $ and multiplying the previous inequality by $e^{-Ct}$, one has
\begin{equation}\label{stima2}
\frac{d}{dt}(e^{-Ct}F(t)) \le  e^{-Ct}(\|\chi_\omega f(t)\|^2_{L^2(Q_{A,1})}+\|h(t)\|^2_{L^2(Q_{A,1})}).
\end{equation}
Integrating \eqref{stima2} over $(0,t)$, for all $t \in [0,T]$ it follows
\[
e^{-Ct}F(t) \le F(0) + \int_0^te^{-C\tau}(\|\chi_\omega f(\tau)\|^2_{L^2(Q_{A,1})}+\|h(\tau)\|^2_{L^2(Q_{A,1})})d\tau.
\]
Hence, for all $t \in [0,T]$,
\[
F(t) \le e^{CT}\left( F(0) + \int_0^T(\|\chi_\omega f(\tau)\|^2_{L^2(Q_{A,1})}+\|h(\tau)\|^2_{L^2(Q_{A,1})})d\tau\right)
\]
and
\begin{equation}\label{stima3}
\sup_{t \in [0,T] } \|y(t)\|^2_{L^2(Q_{A,1})} \le C\left(\|y_0\|^2_{L^2(Q_{A,1})}  + \|f\chi_\omega\|^2_{L^2(Q)}+ \|h\|^2_{L^2(Q)}\right).
\end{equation}
Therefore, by \eqref{stima1}, it follows
\[
\begin{aligned}
\frac{1}{2}\frac{d}{dt}\|y(t)\|^2_{L^2(Q_{A,1})} +\int_0^A\int_0^1ky_x^2dxda& \le \frac{C}{2}\int_0^A\int_0^1y^2(t,a,x)dxda + \frac{1}{2} \int_0^A\int_0^1 \chi_\omega f^2 dx da\\
& +  \frac{1}{2} \int_0^A\int_0^1 h^2 dx da.
\end{aligned}
\]
Integrating over $(0,T)$, we have
\begin{equation}\label{stima3'}
\begin{aligned}
&\frac{1}{2}\|y(T)\|^2_{L^2(Q_{A,1})} +\int_0^T\int_0^A\int_0^1ky_x^2dxdadt\le\frac{1}{2}\|y_0\|^2_{L^2(Q_{A,1})}+ \frac{C}{2}\int_0^T\int_0^A\int_0^1y^2(t,a,x)dxdadt\\
& + \frac{1}{2} \int_0^T\int_0^A\int_0^1 \chi_\omega f^2 dx dadt+\frac{1}{2} \int_0^T\int_0^A\int_0^1 h^2 dx dadt.
\end{aligned}
\end{equation}
Hence, by \eqref{stima3'},
\begin{equation}\label{stima4}
\begin{aligned}
\int_0^T\int_0^A\|\sqrt{k}y_x\|^2_{L^2(0,1)}dadt&\le\|y_0\|^2_{L^2(Q_{A,1})}+ C\int_0^T\|y(t)\|^2_{L^2(Q_{A,1})}dt+ \|\chi_\omega f\|^2_{L^2(Q)}+ \|h\|^2_{L^2(Q)}\\
&\le C\left(  \|y_0\|^2_{L^2(Q_{A,1})}  +\|\chi_\omega f\|^2_{L^2(Q
)} +\|h\|^2_{L^2(Q
)} 
\right)
\end{aligned}
\end{equation}
and \eqref{stimau} follows by \eqref{stima3} and \eqref{stima4} if $y_0 \in L^2(0,A; H^1_k(0,1))$. 

Using the density of $L^2(0,A; H^1_k(0,1))$ in $L^2(Q_{A,1})$, one can prove that \eqref{stimau} holds also if $y_0 \in L^2(Q_{A,1})$.
\subsection{Proof of Lemma \ref{lemma_caccio}}
This section is devoted to the proof of Caccioppoli's inequality \eqref{inequality_caccio}.

Let us consider a smooth function $\xi:[0,1] \rightarrow \mathbb{R}$ defined as follows
$$
\left\{\begin{array}{ll}
0 \leq \xi(x) \leq 1, & \text { for all } x \in[0,1] \\
\xi(x)=1, & x \in \omega^{\prime} \\
\xi(x)=0, & x \in(0,1) \backslash \omega
\end{array}\right.
$$
Since $z$ solves \eqref{adjoint}, then integrating by parts we obtain
\begin{align*}
0 &=\int_{0}^{T} \frac{d}{d t}\left(\int_{0}^{A} \int_{0}^{1}\left(\xi e^{s \psi}\right)^{2} z^{2} d x d a\right) d t \\
&=\int_{Q} 2 s \psi_{t}\left(\xi e^{s \psi}\right)^{2} z^{2}+2\left(\xi e^{s \psi}\right)^{2} z\left(-z_{a}-\left(k z_{x}\right)_{x}+\mu z+ g\right) d x d a d t \\
&=2 s \int_{Q} \psi_{t}\left(\xi e^{s \psi}\right)^{2} z^{2} d x d a d t+2 s \int_{Q} \psi_{a}\left(\xi e^{s \psi}\right)^{2} z^{2} d x d a d t - \int_{Q}\left( k (\xi^{2} e^{2 s \psi})_{x}\right)_x  z^2 d x d a d t \\
&+2 \int_{Q}\left(\xi^{2} e^{2 s \psi} k\right) z_{x}^{2} d x d a d t+2 \int_{Q} \xi^{2} e^{2 s \psi} \mu z^{2} d x d a d t+2 \int_{Q} \xi^{2} e^{2 s \psi} g z d x d a d t .
\end{align*}

Therefore,
\begin{align*}
2 \int_{Q} \xi^{2} e^{2 s \psi} k z_{x}^{2} d x d a d t &=-2 s \int_{Q} \psi_{t}\left(\xi e^{s \psi}\right)^{2} z^{2} d x d a d t-2 s \int_{Q} \psi_{a}\left(\xi e^{s \psi}\right)^{2} z^{2} d x d a d t \\
&+ \int_{Q}\left( k (\xi^{2} e^{2 s \psi})_{x}\right)_x  z^2 d x d a d t-2 \int_{Q} \xi^{2} e^{2 s \psi} \mu z^{2} d x d a d t \\
&-2 \int_{Q} \xi^{2} e^{2 s \psi} g z d x d a d t.
\end{align*}
Hence, taking into account the definition of $\xi$ and the fact that $|\Theta_a| \leq C \Theta^2$, $|\Theta_t| \leq C \Theta^2$, $k \in C^1(\overline{\omega'})$, $\inf\limits_{x\in \omega'} k(x) > 0$ and $\phi \in C^2(\overline{\omega'})$, and applying Young's inequality, we infer that

\begin{align*}
\inf _{\omega^{\prime}}\{k\} \int_{0}^{T} & \int_{0}^{A} \int_{\omega^{\prime}} e^{2 s \psi} z_{x}^{2} d x d a d t  \leq  C s  \int_{0}^{T} \int_{0}^{A} \int_{\omega} \Theta^2 e^{2s \psi} z^{2} d x d a d t \\
&+  C \int_{0}^{T} \int_{0}^{A} \int_{\omega} s^2 \Theta^2 e^{2s \psi} z^{2} d x d a d t+  \int_{0}^{T} \int_{0}^{A} \int_{\omega}  e^{2 s \varphi}  z_{x}^{2} d x d a d t \\
&+\|\mu\|_{L^{\infty}(Q)} \int_{0}^{T} \int_{0}^{A} \int_{\omega} e^{2s \psi} z^{2} d x d a d t+ \int_{0}^{T} \int_{0}^{A} \int_{\omega} e^{2 s \psi} g^{2} d x d a d t \\
& \leq C \int_{0}^{T} \int_{0}^{A} \big( g^2 + s^2 \Theta^2  z^{2}\big) e^{2s \psi} d x d a d t ,
\end{align*}
from which the conclusion follows.


\end{document}